\documentclass[a4paper]{article}

\usepackage{stackengine}
\stackMath

\usepackage{enumitem}
\usepackage{rotating}
\newcommand{\invertediota}{\begin{sideways}\begin{sideways}$\iota$\end{sideways}\end{sideways}}

\usepackage{pdflscape}
\usepackage{mathabx}
\usepackage{philex}
\usepackage{latexsym}

\usepackage{amsthm}
\usepackage{bussproofs}
\usepackage{url}
\usepackage{pxfonts}
\usepackage{microtype}

\usepackage{natbib}

\newtheorem{theorem}{Theorem}
\newtheorem{lemma}{Lemma}

\newtheorem{corollary}{Corollary}
\newtheorem{conjecture}{Conjecture}

\begin{document}
\title{Normalisation for Positive Free Logics without and with Definite Descriptions}

\author{Nils K\"urbis}

\date{}
\maketitle

\begin{center}
Published in the \emph{Review of Symbolic Logic}
\url{https://doi.org/10.1017/S1755020326101166}\bigskip
\end{center}

\begin{abstract}
\noindent This paper proves normalisation theorems for intuitionist and classical positive free logic, without and with the $\invertediota$ operator for definite descriptions `the $F$'. Positive free logic also opens a number of options for rules for $\invertediota$. In total, six different formalisations of theories of definite descriptions will be discussed, three proposed by Lambert, and three alternatives. The latter are motivated by considerations relating to proof-theoretic harmony between introduction and elimination rules. The philosophical importance of the various systems and results is indicated. The paper builds on Kürbis \citeyearpar{kurbisiotanegfreelogic}, but is largely self-contained. The proofs for the present systems are easier than those for negative free logic. 
\end{abstract}


\section{Introduction}
In classical and intuitionist logic, names are assumed to refer and the domain of quantification must contain at least one object. Universally free logic abandons both assumptions. The quantifiers, however, carry existential commitment. Positive free logic arises from a restriction of the quantifiers of classical or intuitionist logic to a predicate interpreted as `exists' or alternatively to instantiating terms that refer. $Ft$ can only be concluded from $\forall xFx$ if $t$ exists or if the name `$t$' refers. Similarly, $\exists xFx$ can only be concluded from $Ft$ if $t$ exists or if the name `$t$' refers. Positive free logic is distinct from negative free logic in that it permits that sentences containing terms that do not refer may nonetheless be true or assertible. This is reflected in a slightly different theory of identity, in fact just the standard theory, and the absence of a rule that lays down that atomic formulas can only be true or assertible if all terms in them refer. Positive free logic arises more directly from classical or intuitionist logic than negative free logic: the only difference is the restriction of the quantifiers. Indeed, the first system of free logic, formalised by Ja\'skowski \citeyearpar{jaskowskirules}, was a positive free logic.\footnote{For overviews, see \citep{bencivengahandbook},  \citep{morschersimonsfreelogic}, \citep{lehmannmorefreelogic}. Scott introduced intuitionist free logic \citep{scottidentityexistence}.} 

Kürbis \citeyearpar{kurbisiotanegfreelogic} proved a normalisation theorem for systems of natural deduction for negative free logic without and with definite descriptions. The present paper does the same for positive free logic. The hard work has been done in the previous paper. The present paper follows its steps and adapts its proofs to suit positive free logic. With fewer rules and a simpler theory of identity than negative free logic, the proofs are simpler. The maximal segments specific to negative free logic either do not arise in positive free logic or such segments are not maximal. The rules for the $\invertediota$ operator for definite descriptions require slight modifications, and it has one rule less. The rest, the rules for the quantifiers and connectives, stay the same. These differences require slight modifications of the proofs and raise philosophical issues on which I'll briefly comment in the following pages. 

A few words on the motivation of positive free logic are in order. Hintikka, Lambert and many free logicians prefer positive over negative free logic. It has also established itself as a logic for quantified modal logic.\footnote{See, e.g., \cite[293ff]{hughescresswell} and \cite[240ff]{garsonmodallogic}. \cite{mendelsohnfitting} use what is effectively a positive free logic that excludes the empty domain of quantification.} Hintikka and Lambert developed free logic with an eye on formalising theories of definite descriptions that are alternatives to Russell's in order to avoid some of what they perceived to be shortcomings of his theory. (See \citep{russellondenoting}, \citep[Introduction, Ch. 3, and Part I, Sec. B, $\ast 14$]{russellwhitehead}, \cite[ch. XVI]{russellintromathphil}, \citep{hintikkaexential}, \citep{hintikkatowardsdefdesc}, \citep{lambertnotesEII}, \citep{lambertnotesEIII}.) A definite description is an expression of the form `the $F$'. According to Russell, a sentence `The $F$ is $G$' is to be analysed as `There is exactly one $F$ and it is $G$'. Hence `the $F$' only appears to be a singular term, an expression purporting to refer to a unique object, as in fact it disappears upon analysis. Outside the context of a complete sentence, says Russell, `the $F$' means nothing \citep[70]{russellwhitehead}. An immediate consequence of Russell's theory is that `The $F$ is $G$' is false, if there is no or more than one $F$. 

Hintikka and Lambert held that definite descriptions should be treated as genuine singular terms and the existence assumptions made by Russell's theory avoided. A definite description `the $F$' is \emph{proper} if there is exactly one $F$, \emph{improper} if there is none or more than one. Hintikka and Lambert held that for the proper definite descriptions, Russell is right: if `the $F$' is proper, `The $F$ is $G$' means that there is exactly one $F$ and it is $G$. They disagreed with Russell over the improper ones: on these, logic should remain almost completely silent. (\cite[83]{hintikkatowardsdefdesc}, \cite[417]{bencivengahandbook}, \cite[69f]{lambertfoundations}) It is not for logic to decide whether `The $F$ is $G$' is true or false if `the $F$' is improper: it may well be that conventions of use determine some such sentences to be true, while Russell's analysis declares them all to be false. 

Negative free logic treats definite descriptions as genuine singular terms, but it is still rather Russellian: when $G$ is an atomic predicate, `The $F$ is $G$' is equivalent to its Russellian analysis. I am not going to assess which of the two theories of definite descriptions is preferable in this paper, but its philosophical discussions will contain material on which such an assessment could draw. 

Systems of natural deduction consist of introduction rules for the derivation of formulas containing the logical expressions and elimination rules for the derivation of consequences of such formulas. A normalisation theorem establishes that any formula that is the conclusion of an introduction rule and major premise of an elimination rule for its main operator may be removed from deductions: these are unnecessary detours, as Gentzen put it \cite[177]{gentzenuntersuchungen}. Normalisation has philosophical significance. According to Dummett, if rules of inference satisfy certain criteria, amongst which normalisation, then they define the meaning of the expression they govern \cite[Ch. 11-13]{dummettLBM}. Normalisation is a necessary condition for rules to count as \emph{harmonious}. (Cf. also Prawitz's Inversion Principle \citep{prawitznaturaldeduction}.) The view is mostly applied to sentence forming operators and quantifiers, with the exception of Tennant (\cite[Sec. 7.10]{tennantnatural}, \citep{tennantabstraction}), who applies it also to term forming operators, amongst which $\invertediota$. The present paper, as the previous one, is a contribution to broadening the proof-theoretic account of the meanings of logical expressions. It will be seen that considerations of harmony point to an imbalance in the rules that formalise Lambert's theory of definite descriptions and suggest alternative rules. In section 8 of this paper alternative theories that are motivated by such considerations are investigated.\footnote{Proof theoretic semantics originates from a fertile remark of Gentzen's \citep[\S 5.13, p.189]{gentzenuntersuchungen}, and was developed in detail by Dummett (in particular in \citep{dummettLBM}), building on Prawitz's normalisation results \citep{prawitznaturaldeduction}. For a brief overview, see \citep{kurbisPTSNM}, for a more detailed survey, see \citep{schroederheistersemantics}.}

\section{Positive Free Logic without and with $\invertediota$}\label{positivefreelogic}
The language is standard, and as in \citep[Sec. 2.1]{kurbisiotanegfreelogic}, and so is the definition of deduction and $\vdash_S$, which essentially follows Troelstra and Schwichtenberg \citeyearpar[Sec. 2.1.1]{troelstraschwichtenberg}. Recall that $A_t^x$ is the formula that results when $x$ is replaced in $A$ by $t$ and that $\exists !$ is the existence predicate. The systems modify rules given by Tennant for negative free logic to suit positive free logic (\cite[Sec. 7.10]{tennantnatural}, \citep{tennantabstraction}). The system \textbf{IPF} of intuitionist positive free logic has the following rules: 

\begin{center}
\AxiomC{$A$} 
\AxiomC{$B$}
\LeftLabel{$(\land  I)$ \ }
\BinaryInfC{$A\land  B$}
\DisplayProof\qquad\qquad 
\AxiomC{$A\land  B$}
\LeftLabel{$(\land  E)$ \ }
\UnaryInfC{$A$}
\DisplayProof\qquad 
\AxiomC{$A\land  B$} 
\UnaryInfC{$B$}
\DisplayProof

\bigskip

\AxiomC{$[A]^i$}
\noLine
\UnaryInfC{$\Pi$}
\noLine
\UnaryInfC{$B$}
\RightLabel{$_i$}
\LeftLabel{$(\rightarrow I)$ \ }
\UnaryInfC{$A\rightarrow B$}
\DisplayProof \qquad \qquad
\AxiomC{$A\rightarrow B$}
\AxiomC{$A$}
\LeftLabel{$(\rightarrow E)$ \ }
\BinaryInfC{$B$}
\DisplayProof\qquad\qquad
\AxiomC{$\bot$}
\LeftLabel{$(\bot E)$ \ }
\UnaryInfC{$B$}
\DisplayProof

\bigskip

\AxiomC{$A$}
\LeftLabel{$(\lor I)$ \ }
\UnaryInfC{$A\lor B$}
\DisplayProof\qquad
\AxiomC{$B$}
\UnaryInfC{$A\lor B$}
\DisplayProof\qquad\qquad
\AxiomC{$A\lor B$}
\AxiomC{$[A]^i$}
\noLine
\UnaryInfC{$\Pi$}
\noLine
\UnaryInfC{$C$}
\AxiomC{$[B]^{j}$}
\noLine
\UnaryInfC{$\Sigma$}
\noLine
\UnaryInfC{$C$}
\RightLabel{$_{i, j}$}
\LeftLabel{$(\lor E)$ \ }
\TrinaryInfC{$C$}
\DisplayProof

\bigskip

\AxiomC{$[\exists !a]^i$}
\noLine
\UnaryInfC{$\Pi$}
\noLine
\UnaryInfC{$A_a^x$}
\LeftLabel{$(\forall I)$ \ }
\RightLabel{$_i$}
\UnaryInfC{$\forall x A$}
\DisplayProof\qquad\qquad
\AxiomC{$\forall xA$}
\AxiomC{$\exists !t$}
\LeftLabel{$(\forall E)$ \ }
\BinaryInfC{$A_t^x$}
\DisplayProof

\bigskip

\AxiomC{$A_t^x$}
\AxiomC{$\exists !t$}
\LeftLabel{$(\exists I)$ \ }
\BinaryInfC{$\exists x A$}
\DisplayProof\qquad\qquad
\AxiomC{$\exists xA$}
\AxiomC{$[A_a^x]^i \ [\exists !a]^j$}
\noLine
\UnaryInfC{$\Pi$}
\noLine
\UnaryInfC{$C$}
\RightLabel{$_{i, j}$}
\LeftLabel{$(\exists E)$ \ }
\BinaryInfC{$C$}
\DisplayProof
\end{center} 

\noindent where in $(\forall I)$ and $(\exists E)$ $a$ does not occur in $\forall xA$, $\exists xA$, $C$, nor in any undischarged assumptions of $\Pi$ except for those in the assumption classes of $A_a^x$ or of $\exists ! a$. 

\begin{center}
\AxiomC{$(=I) \quad \overline{\ t=t\ }$}
\DisplayProof\qquad\qquad
\AxiomC{$t_1=t_2$}
\AxiomC{$A_{t_1}^x$}
\LeftLabel{$(= E)$ \ } 
\BinaryInfC{$A_{t_2}^x$}
\DisplayProof
\end{center}

\noindent Classical positive free logic $\mathbf{CPF}$ has the rules for $\rightarrow$, $\forall$ and $=$ of $\mathbf{IPF}$ and $(\bot E)$ is replaced by: 

\begin{prooftree}
\AxiomC{$[\neg A]^i$}
\noLine
\UnaryInfC{$\Pi$}
\noLine
\UnaryInfC{$\bot$}
\RightLabel{$_i$}
\LeftLabel{$(\bot E_C) \ $}
\UnaryInfC{$A$}
\end{prooftree}

\noindent Vacuous discharged being permitted, $(\bot E)$ is a special case of $(\bot E_C)$, and such applications of the latter are treated as if they were applications of the former. 

Comparing $\mathbf{IPF}$ and $\mathbf{CPF}$ to $\mathbf{INF}$ and $\mathbf{CNF}$ of \citep{kurbisiotanegfreelogic}, note the absence of the rule $(AD)$ of atomic denotation\footnote{$(AD)$: From $Rt_1\ldots t_n\vdash \exists !t_i$, where $R$ is an $n$-place predicate letter (other than $\exists !$) or identity and $1\leq i\leq n$.} and of the premise $\exists !t$ in $(=I)$\footnote{$(=I^n)$: $\exists !t\vdash t=t$.}. This means sequences of application of these rules that gave rise to segments that count as maximal in negative free logic cannot arise (maximal $=$-, $\exists!$- and the first two cases of maximal $(=E)$-segments \citep[Sec. 3.3]{kurbisiotanegfreelogic}). The absence of $(AD)$ also means that sequences of applications of $(=E)$ with minor premises formed from $\exists !$ and a term cannot be removed from deductions and hence are not detours and do not count as maximal. This has some philosophical interest, on which I comment at the end of this section. 

The $\invertediota$ operator is governed by the following rules:\footnote{They modify Tennant's rules and are equivalent to those given by \cite[Ch. 18]{garsonmodallogic}.}

\begin{prooftree} 
\AxiomC{$\exists ! t$}
\AxiomC{$[a=t]^i \ [\exists ! a]^k$}
\noLine
\UnaryInfC{$\Xi$}
\noLine
\UnaryInfC{$F_a^x$}
\AxiomC{$[F_a^x]^j \ [\exists !a]^k$}
\noLine
\UnaryInfC{$\Pi$}
\noLine
\UnaryInfC{$a=t$}
\LeftLabel{$(\invertediota I)$ \ }
\RightLabel{$_{i, j, k}$}
\TrinaryInfC{$\invertediota xF=t$}
\end{prooftree}

\noindent where $a$ is neither free in $F$ nor in $t$ and does not occur in any undischarged assumptions in $\Xi$ and $\Pi$ except those in the assumption classes displayed.

\begin{center} 
\AxiomC{$\invertediota xF=t$}
\AxiomC{$u=t$}
\AxiomC{$\exists !u$}
\LeftLabel{$(\invertediota E_1)$ \ } 
\TrinaryInfC{$F_u^x$}
\DisplayProof

\bigskip

\AxiomC{$\invertediota xF=t$}
\AxiomC{$F_u^x$}
\AxiomC{$\exists ! t$}
\AxiomC{$\exists !u$}
\LeftLabel{$(\invertediota E_2)$ \ }
\QuaternaryInfC{$u=t$}
\DisplayProof
\end{center} 

\noindent Compared to the rules for $\invertediota$ of \citep{kurbisiotanegfreelogic}, note first the existence formulas that may be above the middle premise of $(\invertediota I)$ and that are added as premises in $(\invertediota E_1)$ and $(\invertediota E_2)$. This is all due to the absence of $(AD)$: in negative free logic, these formulas may be derived from identities that occur in the rules. Secondly, $\invertediota$ does not have the third elimination rule that in negative free logic is a special case of $(AD)$. This, too, has some philosophical interest, on which I'll comment at the end of this section. 

$\mathbf{IPF}^{\invertediota}$ is $\mathbf{IPF}$ with $\invertediota$ and its rules added. Similarly for $\mathbf{CPF}^{\invertediota}$. 

Lambert formalises his minimal theory of definite descriptions by adding the following axiom to \textbf{CPF} (\citep[30]{lambertEThe}, \cite[90f]{lambertfoundations}): 

\lbp{LA}{$LA$}{$\forall z (\invertediota xA=z\leftrightarrow \forall y(A_y^x\leftrightarrow y=z))$}

\noindent Call the systems that arise from $\mathbf{IPF}$ and $\mathbf{CPF}$ by adding \rf{LA} $\mathbf{IPF}^{LA}$ and $\mathbf{CPF}^{LA}$. 

\rf{LA} is equivalent to the rules for $\invertediota$ above, showing that they are indeed the intended ones for positive free logic: 

\begin{lemma}\label{LAandrules}
\rf{LA} and the rules for $\invertediota$ are interderivable in $\mathbf{IPF}$. 
\end{lemma}

\begin{proof} It is useful to appeal to rules for the biconditional: 

\begin{center}
\AxiomC{$[B]^i$}
\noLine
\UnaryInfC{$\Pi$}
\noLine
\UnaryInfC{$A$}
\AxiomC{$[A]^j$}
\noLine
\UnaryInfC{$\Pi$}
\noLine
\UnaryInfC{$B$}
\RightLabel{$_{i, j}$}
\LeftLabel{$(\leftrightarrow I)$ \ }
\BinaryInfC{$A\leftrightarrow B$}
\DisplayProof\quad
\AxiomC{$A\leftrightarrow B$}
\AxiomC{$B$}
\LeftLabel{$(\leftrightarrow E_1)$ \ }
\BinaryInfC{$A$}
\DisplayProof\quad
\AxiomC{$A\leftrightarrow B$}
\AxiomC{$A$}
\LeftLabel{$(\leftrightarrow E_2)$ \ }
\BinaryInfC{$B$}
\DisplayProof
\end{center}

\noindent (a) We have:\bigskip

\noindent{\small
\AxiomC{$\forall y(A_y^x\leftrightarrow y=b)$}
\AxiomC{$\exists !a$}
\LeftLabel{$_{(\forall E)}$}
\BinaryInfC{$A_a^x\leftrightarrow a=b$}
\AxiomC{$a=b$}
\LeftLabel{$_{(\leftrightarrow E_1)}$}
\BinaryInfC{$A_a^x$}
\DisplayProof\quad
\AxiomC{$\forall y(A_y^x\leftrightarrow y=b)$}
\AxiomC{$\exists !a$}
\LeftLabel{$_{(\forall E)}$}
\BinaryInfC{$A_a^x\leftrightarrow a=b$}
\AxiomC{$A_a^x$}
\LeftLabel{$_{(\leftrightarrow E_2)}$}
\BinaryInfC{$a=b$}
\DisplayProof
}
\bigskip

\noindent So by $(\invertediota I)$: $(\ast)$ \ $\exists !b, \forall y(A_y^x\leftrightarrow y=b)\vdash \invertediota xA=b$. Furthermore:\bigskip

{\footnotesize
\AxiomC{$\invertediota xA=b$}
\AxiomC{$[a=b]^1$}
\AxiomC{$[\exists !a]^3$}
\LeftLabel{$_{(\invertediota E_1)}$}
\TrinaryInfC{$A_a^x$}
\AxiomC{$\invertediota xA=b$}
\AxiomC{$[A_a^x]^2$}
\AxiomC{$\exists ! b$}
\AxiomC{$[\exists !a]^3$}
\LeftLabel{$_{(\invertediota E_2)}$}
\QuaternaryInfC{$a=b$}
\RightLabel{$_{1, 2}$}
\LeftLabel{$_{(\leftrightarrow I)}$}
\BinaryInfC{$A_a^x\leftrightarrow a=b$}
\LeftLabel{$_{(\forall I)}$}
\RightLabel{$_3$}
\UnaryInfC{$\forall y(A_y^x\leftrightarrow y=b)$}
\DisplayProof
}

\bigskip

\noindent So by $(\leftrightarrow I)$ from $(\ast)$: $\exists !b\vdash \invertediota xA=b\leftrightarrow\forall y(A_y^x\leftrightarrow y=b)$. $(LA)$ follows by $(\forall I)$. 

\noindent (b) Given deductions $\Xi$ of $A_a^x$ from $a=t$ and $\exists !a$ and $\Sigma$ of $a=t$ from $A_a^x$ and $\exists !a$ satisfying the condition of $(\invertediota I)$, derive $\forall y(A_y^x\leftrightarrow y=t)$ by $(\leftrightarrow I)$ and $(\forall I)$, discharging the assumptions listed. Apply $(\forall E)$ with $\exists !t$ to derive $\invertediota xA=t\leftrightarrow \forall y(A_y^x\leftrightarrow y=t)$ from $(LA)$, apply $(\leftrightarrow E_1)$ to derive $\invertediota xA=t$. The derivability of $(\invertediota E_1)$ and $(\invertediota E_2)$ is equally straightforward. 
\end{proof}

\begin{theorem}
$\Gamma\vdash_{\mathbf{IPF}^{LA}} A$ iff $\Gamma\vdash_{\mathbf{IPF}^{\invertediota}} A$, and $\Gamma\vdash_{\mathbf{CPF}^{LA}} A$ iff $\Gamma\vdash_{\mathbf{CPF}^{\invertediota}} A$. 
\end{theorem}

\begin{proof}
Immediate from Lemma \ref{LAandrules}. 
\end{proof}

\noindent $\mathbf{CPF}^{\invertediota}$ is equivalent to Lambert's minimal free theory of definite descriptions $\mathbf{MFD}$. It contains, so to speak, only what is generally agreed upon: if a unique $A$ exists, $\invertediota xA$ refers to it.\footnote{With a caveat, a differing view, possibly more suitable to Meinongians, is proposed in \citep{kurbisiotasequentII}: only uniqueness, not existence, is demanded. Then the right to left hand side of \rf{LA} fails: that there is a unique existing $A$ would not suffice for $\invertediota xA$ to refer to it, as there may also be non-existing $A$s. Kürbis follows Russell, however, in not treating definite descriptions as genuine singular terms. This is the caveat.} If this is not the case, $\mathbf{MFD}$ says nothing specifically to do with definite descriptions about $\invertediota xA$, although it does, of course, say things about $\invertediota xA$ that follow from the laws of positive free logic without $\invertediota$. According to Lambert, $\mathbf{MFD}$ is too minimal as `an important principle evidently not deducible from [\rf{LA}] is the principle of \emph{Cancellation}' \citep[30]{lambertEThe}: 

\lbp{C}{$CA$}{$\invertediota x (x=t)=t$}

\noindent Lambert's system $\mathbf{FD1}$ results from adding \rf{C} to $\mathbf{MFD}$ (\citep[44]{lambertfreedef}, \citep[90f]{lambertfoundations}). 

The strongest of Lambert's systems, $\mathbf{FD2}$, is formalised by adding the following axiom to $\mathbf{CPF}$ (\citep[87]{lambertnotesEIV}, \citep[237]{fraassenlambert}):\footnote{\label{FD2equiv}Lambert also uses the version $\invertediota xA=t\leftrightarrow \forall x (x=t\leftrightarrow (A\land \forall y(A_y^x\rightarrow y=x)))$. They are equivalent as $\exists !a\vdash (A_a^x\land \forall y(A_y^x\rightarrow y=a))\leftrightarrow\forall y(A_y^x\leftrightarrow y=a)$.}

\lbp{FD2}{$FD2$}{$\invertediota xA=t\leftrightarrow \forall x (x=t\leftrightarrow \forall y(A_y^x\leftrightarrow y=x))$}

\noindent \rf{FD2} implies both \rf{LA} and \rf{C}, but is not implied by them. A consequence of \rf{FD2} is 

\lbp{CO}{$CO$}{$(\neg \exists ! t_1\land \neg \exists !t_2)\rightarrow t_1=t_2$}

\noindent Semantically speaking, $\mathbf{FD2}$ has the consequence that, if improper definite descriptions are declared to refer to something, then they all to refer to the same thing.\footnote{In the standard dual domain semantics for positive free logic, \rf{FD2} has the consequence that there is only one object in the outer domain to which all terms refer that do not refer in the inner domain.} Similarly for constants. This is Carnap's \emph{chosen object theory} of `non-denoting' terms, a version of which was also endorsed by Frege in some of his writings. (See \citep[41]{fregesinnbedeutung} and \citep[\S 7, \S 8]{carnapmeaningnecessity}.).\footnote{See \cite{kurbisdefdescr} for overview.} 

After establishing normalisation results for $\mathbf{IPF}$, $\mathbf{CPF}$, $\mathbf{IPF}^{\invertediota}$ and $\mathbf{CPF}^{\invertediota}$, I'll consider $\mathbf{FD1}$ and $\mathbf{FD2}$ and their intuitionist versions. I will then also consider three alternatives. They raise philosophical questions, on which I shall comment briefly in the next and later sections.

\section{Philosophical Comments}
The differences in the rules of positive and negative free logic merit philosophical reflection, but I'll only hint at some thoughts here.\bigskip

\noindent 1. Sequences of applications of $(=E)$ the minor premises of which are existence formulas (maximal $(=E)$-segments of \citep{kurbisiotanegfreelogic}) cannot be removed from deductions in positive free logic. Whereas, due to the normalisation theorem, in negative free logic, only the last two terms in such a sequence must denote for the conclusion of the deduction to be derivable from (some of) its open assumptions, in positive free logic, this is the case for all terms in the sequence. This is a little surprising. From the semantic perspective usually taken by free logicians, positive free logic permits truth without reference even in the case of atomic formulas, in particular identities. Negative free logic does not. In general, negative free logic requires more existence assumptions for truth than positive free logic. Here we have a situation where it is the other way round. When the terms are $\invertediota$ terms, they may even involve predicates that do not occur in the conclusion nor in undischarged assumptions above the last major premise of $(=E)$ in the sequence: thus in positive free logic, the proof requires more conceptual resources that the corresponding proof in negative free logic.\bigskip

\noindent 2. The lacking third elimination rule of $\invertediota$ introduces a certain misbalance that is problematic from the perspective of proof-theoretic semantics and the requirement of harmony. Harmony consists in a balance between introduction and elimination rules for a logical expression. Roughly, the grounds for deriving a formula with a logical expression required by its introduction rules can be retrieved from the formula by its elimination rules. In the rules for $\invertediota$ above, one such ground, the premise $\exists !t$ of $(\invertediota I)$, cannot be retrieved from $\invertediota xF=t$ by an elimination rule. According to proof-theoretic semantics, harmony is necessary for rules governing an expression to define its meaning. It may be that the misbalance can be justified independently, for instance from the nature of reference, or it may be concluded that $\invertediota$ is not an expression the meaning of which is defined by its rules of inference. I'm not going to investigate whether whether such an independent justification can be given. Instead, assuming that the meaning of $\invertediota$ may be defined by the rules governing it, I shall consider two options of remedying the misbalance now and there'll be further reflections on this topic later in this paper.\footnote{\label{perfectbalance}A perfect balance between $(\invertediota I)$ and $(\invertediota E_2)$ suggests permitting the discharge of $\exists !t$ in the right subdeduction. However, $\exists !t$ is derivable from $\exists !a$ and $a=t$, so it is unnecessary to permit its discharge, as it may be derived.}

One option would be to add the third elimination rule for $\invertediota$ of negative free logic to positive free logic, i.e. $\invertediota xF=t\vdash \exists !t$. However, not only would this be against the spirit of positive free logic, it would also be inconsistent. The motivation for positive free logic is that, for which reasons ever, formulas containing constants may be true or assertible even if some terms occurring in them do not refer. This is supposed to extend to singular terms in general. With the additional rule, for any definite description $\invertediota xF$, it follows that $\exists !\invertediota xF$, as from $(=I)$ $\invertediota xF=\invertediota xF$. And so, replacing $t$ and $u$ by $\invertediota xF$ in $(\invertediota E_1)$, $F(\invertediota xF)$, for any $F$. Now replace $F$ by an inconsistent predicate, say $Ax\land\neg Ax$. Then $A(\invertediota x(Ax\land \neg Ax))\land \neg A(\invertediota x(Ax\land \neg Ax))$. Contradiction.

The second option results in an alternative theory of definite descriptions: do not require the premise $\exists !t$ in $(\invertediota I)$ and instead adopt the rule: 

\begin{prooftree} 
\AxiomC{$[a=t]^i \ [\exists ! a]^k$}
\noLine
\UnaryInfC{$\Xi$}
\noLine
\UnaryInfC{$F_a^x$}
\AxiomC{$[F_a^x]^j \ [\exists !a]^k$}
\noLine
\UnaryInfC{$\Pi$}
\noLine
\UnaryInfC{$a=t$}
\LeftLabel{$(\invertediota I')$ \ }
\RightLabel{$_{i, j, k}$}
\BinaryInfC{$\invertediota xF=t$}
\end{prooftree}

\noindent where $a$ is neither free in $F$ nor in $t$ and does not occur in any undischarged assumptions in $\Xi$ and $\Pi$ except those in the assumption classes displayed.\bigskip

\noindent The resulting theory is no longer equivalent to Lambert's. \rf{LA} is derivable from $(\invertediota I')$, $(\invertediota E_1)$ and $(\invertediota E_2)$, but $(\invertediota I')$ is not derivable from \rf{LA}. The new set of rules is equivalent to two axioms: 

\lbp{LA1}{$HA^{rl}$}{For any term $t$: $\forall y(A_y^x\leftrightarrow t=y)\rightarrow\invertediota xA=t$}

\lbp{LA2}{$LA^{lr}$}{$\forall z (\invertediota xA=z\rightarrow \forall y(A_y^x\leftrightarrow z=y))$}

\noindent The label \rf{LA1} was chosen because this axiom is the right to left direction of an axiom I'm naming after Hintikka in a later section. 

\rf{LA1} says that if $t$ is the unique existing $A$, then it is the $A$. An easy consequence is $\forall x(A\leftrightarrow x=t)\vdash \exists xA\leftrightarrow \exists !t$, hence $\forall x(A\leftrightarrow x=t), \exists !t\vdash A_t^x$. Furthermore, letting $A$ be $x=t$, $\vdash\invertediota x (x=t)=t$, which is the second axiom of Lambert's preferred free theory of definite descriptions, on which more below. 

Call the theory of definite descriptions axiomatised by \rf{LA1} and \rf{LA2} $\mathbf{FDA}$. In section 8 of this paper it is observed that a normalisation theorem for the system resulting by replacing $(\invertediota I)$ by $(\invertediota I')$ can be proved.\bigskip

\noindent 3. Questions regarding proof-theoretic semantics and harmony also motivate the discussion of systems of natural deduction equivalent to Lambert's $\mathbf{FD1}$ and $\mathbf{FD2}$ and two further alternative theories of definite descriptions later in this paper. One will be based on the axiom of Hintikka's already mentioned, the other on an axiom that contained a typo in a paper of Lambert's. While harmony is problematic in the former two systems, the rules of the latter two do appear to be in harmony. This raises a number of philosophical questions.

\section{Simplifying Lemmas and Preliminaries}
I'll repeat definitions, lemmas, theorems and corollaries, the latter without proof, from \citep{kurbisiotanegfreelogic} for convenience, distinguishing them by an \textbf{N} from those of the present paper. 

Definitions 1 and 2 (of $\vdash_S$) of \citep{kurbisiotanegfreelogic} stay the same. Hintikka's Law  $\exists !t\leftrightarrow \exists x \ x=t$ is derivable in positive free logic. As their proofs do not appeal to $(AD)$ or $(=I^n)$, Lemmas 1 and 2 of \citep{kurbisiotanegfreelogic} carry over to $\mathbf{IPF}$, $\mathbf{IPF}^{\invertediota}$ and $\mathbf{CPF}$, and as before, in $\mathbf{CPF}^{\invertediota}$, $(\bot E_C)$ can be restricted only to atomic, but not prime, conclusions.\medskip 

\noindent \textbf{Definition 1N}. \emph{Prime formulas} are those formed from parameters and constants by predicate letters, $\exists!$ or =. \emph{Atomic formulas} are those formed from terms by these expressions. 

\noindent \textbf{Lemma 1N}. $(=E)$ may be restricted to atomic conclusions. 

\noindent \textbf{Lemma 2N}. $(\bot E)$ may be restricted to prime conclusions. 

\medskip 

\noindent Applications of $(=E)$, $(\bot E)$ and $(\bot E_C)$ are from now of assumed to be restricted as stated. 

Definitions 3, 4, 5 and 6 carry over unchanged, and those of Definition 5 are all the maximal segments that need to be considered. Definitions 12 and 13 stay the same, too.\medskip

\noindent \textbf{Definition 3N}. The \emph{major premise} of an elimination rule is the formula that displays the connective or $\invertediota$ in the general statement of the rule, here always the leftmost premise. All others are \emph{minor premises}. 

\noindent \textbf{Definition 4N}. A \emph{segment} is a sequence of formulas $C_1\ldots C_n$ such that $C_1$ is not the conclusion of $(\lor E)$ or $(\exists E)$, $C_n$ is not a minor premise of $(\lor E)$ or $(\exists E)$, and if $n>1$ then for all $i<n$, $C_i$ is a minor premise of $(\lor E)$ or $(\exists E)$, $C_{i+1}$ the conclusion. $n$ is the \emph{length} of the segment. 

\noindent \textbf{Definition 5N}. A segment is \emph{maximal} if its last formula is the major premise of an elimination rule, and if its length is $1$, it is the conclusion of an introduction rule. 

\noindent \textbf{Definition 6N}. (a) If $A$ is a prime formula $Rt_1\ldots t_n$, then $d(A)=0$ if $R$ is a predicate letter, and $d(A)=1$ if $R$ is $\exists !$ or $=$. (b) $d(t)=0$, if $t$ is an atomic term, $d(F)+1$ if $t$ is $\invertediota xF$. (c) If $A$ is an atomic formula $Rt_1\ldots t_n$, then $d(A)=d(t_1)+\ldots d(t_n)$ if $R$ is a predicate letter, and $d(A)=d(t_1)+\ldots d(t_n)+1$ if $R$ is $\exists!$ or $=$. (d) If $A$ is: 

(i) $\bot$, then $d(A)=1$; 

(ii) $\neg B$, then $d(A)=d(B)+1$; 

(iii) $B\land C$, $B\supset C$ or $B\lor C$, then $d(A)=d(B)+d(C)+1$; 

(iv) $\forall xB$ or $\exists xB$, then $d(A)=d(B)+2$

\noindent \textbf{Definition 12N}. A deduction is in \emph{normal form} if it contains no maximal segments.

\noindent \textbf{Definition 13N}. The \emph{rank} of a deduction is the pair $\langle d, l\rangle$, where $d$ is the highest degree of a maximal I/E-segment or $0$ if there is none, and $l$ is the sum of the lengths of I/E maximal segments of highest degree. $\langle d, l\rangle < \langle d', l'\rangle$ iff either (i) $d<d'$ or (ii) $d=d'$ and $l<l'$. 
\medskip

\noindent Maximal I/E-segments are those arising from introduction and elimination rules, which are all the maximal segments in positive free logic. 

Definitions 7 to 16, Lemma 3 and Corollaries 1 to 6 to Theorem 1, the normalisation theorem for $\mathbf{INF}$, of \citep{kurbisiotanegfreelogic} were concerned with maximal segments specific to negative free logic, the form of normal deduction and the subformula propery. For brevity I won't consider the subformula property in the present paper, but, due to the standard theory of identity and the absence of $(AD)$, it is rather simpler in positive than in negative free logic. As before it would be possible reformulate the systems in Ja\'skowski-style to improve on the subformula property \citep[Sec. 5]{kurbisiotanegfreelogic}. Doing so for $\mathbf{CPF}$ would result in a system equivalent to Ja\'skowski's \citeyearpar{jaskowskirules}. I won't consider this option here.

\section{Normalisation for Positive Free Logic without Definite Descriptions}
Normalisation for deductions in $\mathbf{IPF}$ follows straightforwardly from normalisation for those in $\mathbf{INF}$ \citep[Sec. 4.1]{kurbisiotanegfreelogic}. The reduction procedures for the propositional connectives and the quantifiers are the same, the former as given by Prawitz, the latter repeated here for convenience. Replace the inferences on the left by those on the right: 

\begin{center}
\AxiomC{$[\exists !a]^i$}
\noLine
\UnaryInfC{$\Pi$}
\noLine
\UnaryInfC{$A_a^x$}
\UnaryInfC{$\forall xA$}
\AxiomC{$\Sigma$}
\noLine
\UnaryInfC{$\exists !t$}
\RightLabel{$_i$}
\BinaryInfC{$A_t^x$}
\DisplayProof\qquad$\leadsto$\qquad
\AxiomC{$\Sigma$}
\noLine
\UnaryInfC{$[\exists !t]$}
\noLine
\UnaryInfC{$\Pi_t^a$}
\noLine
\UnaryInfC{$A_t^x$}
\DisplayProof

\bigskip

\AxiomC{$\Xi$}
\noLine
\UnaryInfC{$A_t^x$}
\AxiomC{$\Sigma$}
\noLine
\UnaryInfC{$\exists !t$}
\BinaryInfC{$\exists x A$}
\AxiomC{$[A_a^x]^i\ [\exists !a]^j$}
\noLine
\UnaryInfC{$\Pi$}
\noLine
\UnaryInfC{$C$}
\RightLabel{$_{i, j}$}
\BinaryInfC{$C$}
\DisplayProof\qquad$\leadsto$\qquad
\def\defaultHypSeparation{\hskip .01in}
\AxiomC{$\Xi$}
\noLine
\UnaryInfC{$[A_t^x]$}
\AxiomC{$\Sigma$}
\noLine
\UnaryInfC{$[\exists !t]$}
\noLine
\BinaryInfC{$\Pi_t^a$}
\noLine
\UnaryInfC{$C$}
\DisplayProof
\end{center}

\noindent Applications of $(=E)$ with major premise of form $t=t$ remain vacuous and are simply deleted from deductions. 

\begin{theorem}\label{normalIPF}
Deductions in $\mathbf{IPF}$ can be brought into normal form. 
\end{theorem}

\begin{proof}
The theorem follows by a standard induction: if suitably chosen, applying a reduction procedure reduces the rank of a deduction. The theorem also follows from Theorem 1 of \citep{kurbisiotanegfreelogic}, because a deduction in $\mathbf{IPF}$ is a deduction of $\mathbf{INF}$ in which $(AD)$ is not applied and any formula $t=t$ is either an assumption or derived by an elimination rule.
\end{proof}

\noindent Corresponding to Theorem 5 of \citep[Sec. 7]{kurbisiotanegfreelogic} is the following, which is also proved by a standard, but slightly simpler, induction over the rank of deductions: 

\begin{theorem}\label{normalCPF}
Deductions in $\mathbf{CPF}$ can be brought into normal form. 
\end{theorem}

\section{Normalisation for Positive Free Logic with Definite Descriptions}
With definite descriptions, arbitrarily complex terms may be used as the instantiating terms in $(\exists I)$ and $(\forall E)$. This is a problem for normalisation, as in the reduction procedures for the quantifiers parameters are replaced by the instantiating terms, which increases the complexity of formulas if the term is complex. As in the case of negative free logic, the problem is solved by adding Indrzejczak's rule $(=I^n)$ to positive free logic \cite[Sec. 4]{andrzejcutfreefreelogic}: 

\begin{prooftree}
\AxiomC{$\exists !t$}
\AxiomC{$[a=t]^i$}
\noLine
\UnaryInfC{$\Pi$}
\noLine
\UnaryInfC{$C$}
\LeftLabel{$(=I^{nG})$}
\RightLabel{$_i$}
\BinaryInfC{$C$}
\end{prooftree}

\noindent where the parameter $a$ occurs neither in $t$, nor in $C$, nor in any open assumptions of $\Pi$ other than those of the assumption class of $a=t$.\bigskip

\noindent As Hintikka's Law is valid, the rule is derivable. $(=I)$ is not derivable from $(=I^{nG})$\footnote{I shan't prove this here, but the claim follows from the normalisation theorems: to derive $(=I)$ would require a deduction of $t=t$ from no assumptions, that is, a proof, and any proof in normal form ends in an application of an introduction rule. Hence any proof of $t=t$ ends in $(=I)$ and therefore cannot consists of anything else.}, so the systems $\mathbf{IPF}{'}^{\invertediota}$ and $\mathbf{CPF}{'}^{\invertediota}$ are  $\mathbf{IPF}^{\invertediota}$+$(=I^{nG})$ and $\mathbf{CPF}^{\invertediota}$+$(=I^{nG})$ respectively. 

The proof of Lemma 7 of \citep{kurbisiotanegfreelogic} appeals to $(AD)$, so requires a slightly different proof: 

\begin{lemma}\label{quantifiersrestricted}
Given $(=I^{nG})$, $(\forall E)$ and $(\exists I)$ may be restricted to atomic terms. 
\end{lemma}

\begin{proof}
Replace an application of $(\forall E)$ in which $t$ is complex by the following: 

\begin{prooftree}
\AxiomC{$\exists !t$}
\AxiomC{$[a=t]^i$}
\AxiomC{$\forall xA$}
\AxiomC{$[a=t]^i$}
\AxiomC{$\exists !t$}
\LeftLabel{$_{(=E)}$}
\BinaryInfC{$\exists !a$}
\LeftLabel{$_{(\forall E)}$}
\BinaryInfC{$A_a^x$}
\LeftLabel{$_{(=E)}$}
\BinaryInfC{$A_t^x$}
\LeftLabel{$_{(=I^{nG})}$}
\RightLabel{$_i$}
\BinaryInfC{$A_t^x$}
\end{prooftree}

\noindent Similarly for $(\exists I)$. 
\end{proof}

\noindent Note that the deduction contains what in negative free logic counts as a maximal segment, i.e. the application of $(=E)$: there is no way around this in positive free logic, because we cannot deduce the minor premise $\exists !a$ of $(\forall E)$ from $a=t$ alone, and no other assumptions are given from which it could be derived. 

\bigskip

\noindent \textsc{Comment}. As in the case of negative free logic, we may ask a philosophical question regarding the addition of $(=I^{nG})$. As noted in the previous paper in response to a referee's comment, this rule hides a maximal formula. In positive free logic, this may look even worse, as it is an additional rule, not one that permits economies by replacing another rule. Further discussion must be left for another occasion. 

\bigskip

\noindent As the rules for $\invertediota$ are slightly different from those for negative free logic, the reduction procedures are also slightly different. They are as follows:\bigskip 

\noindent (1) $(\invertediota I)$ followed by $(\invertediota E_1)$: 
{\footnotesize
\begin{center} 
\AxiomC{$\Sigma_1$}
\noLine
\UnaryInfC{$\exists ! t$}
\AxiomC{$[a=t]^i \ [\exists !a]^k$}
\noLine
\UnaryInfC{$\Xi$}
\noLine
\UnaryInfC{$F_a^x$}
\AxiomC{$[F_a^x]^j \ [\exists !a]^k$}
\noLine
\UnaryInfC{$\Pi$}
\noLine
\UnaryInfC{$a=t$}
\RightLabel{$_{i, j, k}$}
\TrinaryInfC{$\invertediota xF=t$}
\AxiomC{$\Sigma_2$}
\noLine
\UnaryInfC{$u=t$}
\AxiomC{$\Sigma_3$}
\noLine
\UnaryInfC{$\exists !u$}
\TrinaryInfC{$F_u^x$}
\DisplayProof
\quad $\leadsto$ \quad
\AxiomC{$\mathbin{\stackon[6pt]{[u=t]}{\Sigma_2}} \ \mathbin{\stackon[6pt]{[\exists !u]}{\Sigma_3}}$}
\alwaysNoLine
\UnaryInfC{$\Xi_u^a$}
\UnaryInfC{$F_u^x$}
\DisplayProof
\end{center}
}
\noindent (2) $(\invertediota I)$ followed by $(\invertediota E_2)$: 
{\footnotesize
\begin{center}
\AxiomC{$\Sigma_1$}
\noLine
\UnaryInfC{$\exists ! t$}
\AxiomC{$[a=t]^i \ [\exists !a]^k$}
\noLine
\UnaryInfC{$\Xi$}
\noLine
\UnaryInfC{$F_a^x$}
\AxiomC{$[F_a^x]^j \ [\exists !a]^k$}
\noLine
\UnaryInfC{$\Pi$}
\noLine
\UnaryInfC{$a=t$}
\RightLabel{$_{i, j, k}$}
\TrinaryInfC{$\invertediota xF=t$}
\AxiomC{$\Sigma_2$}
\noLine
\UnaryInfC{$F_u^x$}
\AxiomC{$\Sigma_3$}
\noLine
\UnaryInfC{$\exists ! t$}
\AxiomC{$\Sigma_4$}
\noLine
\UnaryInfC{$\exists ! u$}
\QuaternaryInfC{$u=t$}
\DisplayProof\quad$\leadsto$\quad
\alwaysNoLine
\AxiomC{$\mathbin{\stackon[6pt]{[F_u^x]}{\Sigma_2}} \ \mathbin{\stackon[6pt]{[\exists !u]}{\Sigma_4}}$}
\UnaryInfC{$\Pi_u^a$}
\UnaryInfC{$u=t$}
\DisplayProof
\end{center}
}
\bigskip

\noindent \textsc{Comment.} In the second reduction procedure the minor premise $\exists !t$ of $(\invertediota E_2)$ is not used. We could reformulate $(\invertediota I)$ so as to permit the discharge of $\exists !t$ above the rightmost subdeduction, and then the minor premise would be used in the reduction procedure. However, as $\exists !t$ follows from $a=t$ and $\exists !a$, there is no need to do so, and the rule as stated is slightly more economical.\footnote{Compare with the remark in footnote \ref{perfectbalance}.}\bigskip

\noindent As in negative free logic, we face the problem that $u$ may be complex, and so replacing it for the parameter $a$ may increase the complexity of formulas. Once more an observation of Indrzejczak's solves the problem \citep{andrzejrussellian}: the rules for $\invertediota$ may be restricted to one occurrence of an $\invertediota$ term, where it is displayed in the rules. Lemma 8 of \citep{kurbisiotanegfreelogic} appeals to $(AD)$, so the corresponding result for positive free logic requires a slightly different proof. 

\begin{lemma}\label{iotarestricted}
Given $(=I)$, $(=E)$ and $(=I^{nG})$, $(\invertediota I)$, $(\invertediota E_1)$ and $(\invertediota E_2)$ can be restricted to the single occurrence of an $\invertediota$ term where it is displayed in the rules. 
\end{lemma}

\begin{proof} Double lines mark inferences by symmetry of identity, which is derivable from $(=I)$ and $(=E)$. 

\noindent (1) Replace an application of $(\invertediota I)$ where $t$ is a complex term $\invertediota yG$ by: 
{\scriptsize
\begin{prooftree}
\def\defaultHypSeparation{\hskip .05in}
\AxiomC{$\exists !\invertediota yG$}
\AxiomC{$[b=\invertediota yG]^l$}
\AxiomC{$[b=\invertediota yG]^l$}
\AxiomC{$\exists ! \invertediota yG$}
\LeftLabel{$_{(=E)}$}
\BinaryInfC{$\exists !b$}
\AxiomC{[$a=b]^i$}
\AxiomC{$[b=\invertediota yG]^l$}
\LeftLabel{$_{(=E)}$}
\BinaryInfC{$[a=\invertediota yG]$}
\noLine
\UnaryInfC{$\Xi$}
\noLine
\UnaryInfC{$F_a^x$}
\AxiomC{$[F_a^x]^j \ [\exists !a]^k$}
\noLine
\UnaryInfC{$\Pi$}
\noLine
\UnaryInfC{$a=\invertediota yG$}
\doubleLine
\UnaryInfC{$\invertediota yG=a$}
\AxiomC{$[b=\invertediota yG]^l$}
\doubleLine
\UnaryInfC{$\invertediota yG=b$}
\LeftLabel{$_{(=E)}$}
\BinaryInfC{$a=b$}
\LeftLabel{$_{(\invertediota I)}$}
\RightLabel{$_{i, j, k}$}
\TrinaryInfC{$\invertediota xF=b$}
\LeftLabel{$_{(=E)}$}
\BinaryInfC{$\invertediota xF=\invertediota yG$}
\RightLabel{$_l$}
\LeftLabel{$_{(=I^{nG})}$}
\BinaryInfC{$\invertediota xF=\invertediota yG$}
\end{prooftree} 
}
\noindent where $a$ and $b$ are fresh parameters. 

\noindent (2) Replace an application of $(\invertediota E_1)$ where $t$ is a term $\invertediota yG$ and $u$ atomic by: 
{\scriptsize
\begin{prooftree}
\def\defaultHypSeparation{\hskip .1in}
\AxiomC{$u=\invertediota yG$}
\AxiomC{$\exists !u$}
\LeftLabel{$_{(=E)}$}
\BinaryInfC{$\exists !\invertediota yG$}
\AxiomC{$\invertediota xF=\invertediota yG$}
\doubleLine
\UnaryInfC{$\invertediota yG=\invertediota xF$}
\AxiomC{$[a=\invertediota yG]^i$}
\doubleLine
\UnaryInfC{$\invertediota yG=a$}
\LeftLabel{$_{(=E)}$}
\BinaryInfC{$\invertediota xF=a$}
\AxiomC{$[a=\invertediota yG]^i$}
\doubleLine
\UnaryInfC{$\invertediota yG=a$}
\AxiomC{$u=\invertediota yG$}
\LeftLabel{$_{(=E)}$}
\BinaryInfC{$u=a$}
\AxiomC{$\exists !u$}
\LeftLabel{$_{(\invertediota E_1)}$}
\TrinaryInfC{$F_u^x$}
\RightLabel{$_i$}
\LeftLabel{$_{(=I^{nG})}$}
\BinaryInfC{$F_u^x$}
\end{prooftree}
}
\noindent where $a$ is a fresh parameter. 

If $u$ is a complex term $\invertediota zH$, replace $u$ by a fresh parameter $b$; deduce $b=\invertediota yG$ from $\invertediota zH=\invertediota yG$ and $b=\invertediota zH$, $\exists !b$ from $b=\invertediota zH$ and $\exists !\invertediota zH$, and $F_{\invertediota zH}^x$ from $b=\invertediota zH$ and $F_b^x$, all by $(=E)$; discharge $b=\invertediota zH$ by $(=I^{nG})$ with premise $\exists !\invertediota zH$. 

\noindent (3) Replace an application of $(\invertediota E_2)$ where $t$ is a term $\invertediota yG$ and $u$ atomic by: 
{\scriptsize
\begin{prooftree}
\def\defaultHypSeparation{\hskip .1in}
\AxiomC{$\exists !\invertediota yG$}
\AxiomC{$[a=\invertediota yG]^i$}
\AxiomC{$[a=\invertediota yG]^i$}
\doubleLine
\UnaryInfC{$\invertediota yG=a$}
\AxiomC{$\invertediota xF=\invertediota yG$}
\LeftLabel{$_{(=E)}$}
\BinaryInfC{$\invertediota xF=a$}
\AxiomC{$F_u^x$}
\AxiomC{$[a=\invertediota yG]^i$}
\doubleLine
\UnaryInfC{$\invertediota yG=a$}
\AxiomC{$\exists !\invertediota yG$}
\LeftLabel{$_{(=E)}$}
\BinaryInfC{$\exists !a$}
\AxiomC{$\exists ! u$}
\LeftLabel{$_{(\invertediota E_2)}$}
\QuaternaryInfC{$u=a$}
\LeftLabel{$_{(=E)}$}
\BinaryInfC{$u=\invertediota yG$}
\RightLabel{$_i$}
\LeftLabel{$_{(=I^{nG})}$}
\BinaryInfC{$u=\invertediota yG$}
\end{prooftree}
}
\noindent where $a$ is a fresh parameter. 

If $u$ is a complex term $\invertediota zH$, replace $u$ by a fresh parameter $b$; deduce $F_b^x$ from $b=\invertediota zH$ and $F_{\invertediota zH}^x$, $\exists !b$ from $b=\invertediota zH$ and $\exists !\invertediota zH$, and $\invertediota zH=\invertediota yG$ from $b=\invertediota zG$ and $b=\invertediota zH$, all by $(=E)$; discharge $b=\invertediota zH$ by $(=I^{nG})$ with premise $\exists !\invertediota zH$. 
\end{proof}

\noindent Applications of the $\invertediota$ rules in $\mathbf{IPF}{'}^{\invertediota}$ and $\mathbf{CPF}{'}^{\invertediota}$ are therefore restricted accordingly to only one occurrence of an $\invertediota$ term where it is displayed in the rules. The following is immediate from the foregoing: 

\begin{theorem}\label{equivalencerestrictions}
(a) $\Gamma\vdash_{\mathbf{IPF}^{\invertediota}} A$ iff $\Gamma\vdash_{\mathbf{IPF}{'}^{\invertediota}}A$. (b) $\Gamma\vdash_{\mathbf{CPF}^{\invertediota}} A$ iff $\Gamma\vdash_{\mathbf{CPF}{'}^{\invertediota}}A$
\end{theorem}

\begin{proof}
From the derivability of $(=I^{nG})$ and Lemmas \ref{quantifiersrestricted} and \ref{iotarestricted}. 
\end{proof}

\noindent Definition 4N  is modified as in Definition 18 of \citep{kurbisiotanegfreelogic}: insert `or $(=I^{nG})$' after $(\exists E)$. Definition 5N stays the same. 

\begin{theorem}\label{normalIPFDD}
Deductions in $\mathbf{IPF}{'}^{\invertediota}$ can be brought into normal form. 
\end{theorem}

\begin{proof}
By a standard induction: applying the reduction procedures to a suitably chosen maximal segment reduces the rank of the deduction. Because of the slightly different rules for $\invertediota$, deductions in $\mathbf{IPF}{'}^{\invertediota}$ are not special cases of deductions in $\mathbf{INF}{'}^{\invertediota}$. However, the slightly different reduction procedures pose no new problems, and so the proof proceeds essentially as the proof of Theorem 4 of \citep{kurbisiotanegfreelogic}. 
\end{proof}

\noindent Concerning $\mathbf{CPF}$, Definitions 20 and 22 of \citep{kurbisiotanegfreelogic} stay the same, Definition 21 contains a redundancy `or the premise of $(AD)$', and Lemma 10 goes through as before:

\medskip

\noindent \textbf{Definition 20N}. A \emph{segment} is a sequence of formulas arising from applications of $(=I^{nG})$ as in Definition 4N or a sequence of formulas $A_1\ldots A_n$ such that $A_1$ is not the conclusion of $(\bot E_C)$, and for all $i$, $A_i$ is the minor premise of $(\rightarrow E)$ the major premise of which is discharged by $(\bot E_C)$, $A_{i+i}$ is the conclusion of that application of $(\bot E_C)$, and $A_n$ is not a minor premise of $(\rightarrow E)$ the major premise of which is discharged by $(\bot E_C)$. 

\noindent \textbf{Definition 21N}. Add the following at the end of Definition 5N of maximal segment: `or a segment the last formula of which is the conclusion of $(\bot E_C)$ and the major premise of an elimination rule or the premise of $(AD)$'. 

\noindent \textbf{Definition 22N}. An assumption discharged by $(\bot E_C)$  is \emph{regular} if it is the major premise of $(\rightarrow E)$. A proof is \emph{regular} if all assumptions discharged by $(\bot E_C)$ in it are regular. 

\noindent \textbf{Lemma 10N}. (a) Any proof can be transformed into a regular proof. (b) In a regular proof, any assumption discharged by $(\bot E_C)$ stands to the right of a formula on a $(\bot E_C)$-segment. 

\medskip

\noindent The slightly different rules for $\invertediota$ require slightly a different adaptation of Andou's method to prove normalisation \citep{andounormalisation}. The reduction procedures are as follows:\bigskip 

\noindent (1) $(\bot E_C)$ followed by $(\invertediota E_1)$:
 
{\scriptsize
\begin{center}
\def\defaultHypSeparation{\hskip .1in}
\AxiomC{$[\neg \invertediota xF=t]^i$}
\AxiomC{$\Xi$}
\noLine
\UnaryInfC{$\invertediota xF=t$}
\BinaryInfC{$\bot$}
\noLine
\UnaryInfC{$\Pi$}
\noLine
\UnaryInfC{$\bot$}
\RightLabel{$_i$}
\UnaryInfC{$\invertediota xF=t$}
\AxiomC{$\Sigma_1$}
\noLine
\UnaryInfC{$u=t$}
\AxiomC{$\Sigma_2$}
\noLine
\UnaryInfC{$\exists !u$}
\TrinaryInfC{$F_u^x$}
\DisplayProof\quad $\leadsto$\quad  
\AxiomC{$[\neg F_u^x]^j$}
\AxiomC{$\Xi$}
\noLine
\UnaryInfC{$\invertediota xF=t$}
\AxiomC{$\Sigma_1$}
\noLine
\UnaryInfC{$u=t$}
\AxiomC{$\Sigma_2$}
\noLine
\UnaryInfC{$\exists !u$}
\TrinaryInfC{$F_u^x$}
\BinaryInfC{$\bot$}
\noLine
\UnaryInfC{$\Pi$}
\noLine
\UnaryInfC{$\bot$}
\RightLabel{$_j$}
\UnaryInfC{$F_u^x$}
\DisplayProof
\end{center}
}

\noindent (2) $(\bot E_C)$ followed by $(\invertediota E_2)$: 

{\scriptsize
\begin{center}
\def\defaultHypSeparation{\hskip .05in}
\AxiomC{$[\neg \invertediota xF=t]^i$}
\AxiomC{$\Xi$}
\noLine
\UnaryInfC{$\invertediota xF=t$}
\BinaryInfC{$\bot$}
\noLine
\UnaryInfC{$\Pi$}
\noLine
\UnaryInfC{$\bot$}
\RightLabel{$_i$}
\UnaryInfC{$\invertediota xF=t$}
\AxiomC{$\Sigma_1$}
\noLine
\UnaryInfC{$F_u^x$}
\AxiomC{$\Sigma_2$}
\noLine
\UnaryInfC{$\exists !t$}
\AxiomC{$\Sigma_3$}
\noLine
\UnaryInfC{$\exists !u$}
\QuaternaryInfC{$u=t$}
\DisplayProof\ $\leadsto$\ 
\AxiomC{$[\neg u=t]^j$}
\AxiomC{$\Xi$}
\noLine
\UnaryInfC{$\invertediota xF=t$}
\AxiomC{$\Sigma_1$}
\noLine
\UnaryInfC{$F_u^x$}
\AxiomC{$\Sigma_2$}
\noLine
\UnaryInfC{$\exists !t$}
\AxiomC{$\Sigma_3$}
\noLine
\UnaryInfC{$\exists !u$}
\QuaternaryInfC{$u=t$}
\BinaryInfC{$\bot$}
\noLine
\UnaryInfC{$\Pi$}
\noLine
\UnaryInfC{$\bot$}
\RightLabel{$_j$}
\UnaryInfC{$u=t$}
\DisplayProof
\end{center}
}

\begin{theorem}\label{normalCPFDD}
Deductions in $\mathbf{CPF}{'}^{\invertediota}$ can be brought into normal form. 
\end{theorem}

\begin{proof}
From lemma 10N and a standard induction: applying the reduction procedures to suitably chosen maximal segments reduces the rank of deductions. The proof proceeds essentially as in the proof of Theorem 7 of \citep{kurbisiotanegfreelogic}. 
\end{proof}

\section{Lambert's Stronger Systems}
\subsection{FD1}
Call the systems that result from $\mathbf{IPF}{'}^{\invertediota}$ and $\mathbf{CPF}{'}^{\invertediota}$ by adding \rf{C} (with an inference line over it to indicate its derivation from no premisses) $\mathbf{IPF}{'}^{\invertediota 1}$ and $\mathbf{CPF}{'}^{\invertediota 1}$. These systems raises interesting questions regarding harmony. 

\rf{C} is an introduction rule for $\invertediota$: it licences the derivation of identities flanked by an $\invertediota$ term, as does $(\invertediota I)$. It also licences new inferences, in particular, its instances, which were not derivable in $\mathbf{IPF}{'}^{\invertediota}$ and $\mathbf{CPF}{'}^{\invertediota}$. But the new formulas are never derived at the cost of maximal formulas: 

\begin{theorem}
Deductions in $\mathbf{IPF}{'}^{\invertediota 1}$ and $\mathbf{CPF}{'}^{\invertediota 1}$ can be brought into normal form. 
\end{theorem}

\begin{proof}
\emph{Mutatis mutandis} the proof proceeds as for Theorems \ref{normalIPFDD} and \ref{normalCPFDD}, by noting the following. (1) Using \rf{C} as the major premise in $(\invertediota E_1)$ and $(\invertediota E_2)$ is redundant: 

\begin{center} 
\AxiomC{$\overline{\ \invertediota x (x=t) = t\ }$}
\AxiomC{$\Sigma$}
\noLine
\UnaryInfC{$u=t$}
\AxiomC{$\Pi$}
\noLine
\UnaryInfC{$\exists !u$}
\LeftLabel{$_{(\invertediota E_1)}$} 
\TrinaryInfC{$u=t$}
\DisplayProof

\bigskip

\AxiomC{$\overline{\ \invertediota x (x=t) = t\ }$}
\AxiomC{$\Sigma$}
\noLine
\UnaryInfC{$u=t$}
\AxiomC{$\Pi$}
\noLine
\UnaryInfC{$\exists ! t$}
\AxiomC{$\Xi$}
\noLine
\UnaryInfC{$\exists !u$}
\LeftLabel{$_{(\invertediota E_2)}$}
\QuaternaryInfC{$u=t$}
\DisplayProof
\end{center} 

\noindent In both cases, keep only $\Sigma$ and continue the deduction with the rule applied to the lower $u=t$. Due to the restriction on $(\invertediota E_1)$ and $(\invertediota E_2)$, $t$ is an atomic term, so this is the only case that needs to be considered. (2) Deriving $\invertediota x (x=t)=t$ by $(\bot E_C)$ is also redundant. Anything above this conclusion can be deleted and $\invertediota x (x=t)=t$ derived by \rf{C} instead. If the conclusion of $(\bot E_C)$ is $t=\invertediota x(x=t)$, derive it from \rf{C} by symmetry, deleting all the rest. The situation considered in Andou's method to prove normalisation for the classical systems then reduces to the two cases displayed above. 
\end{proof}

\noindent \rf{C} affords new grounds for the derivation of formulas with $\invertediota$ terms. It does not, however, afford any new consequences of such formulas that are specifically to do with $\invertediota$. This points to a misbalance and a lack of harmony: $(\invertediota E_1)$ and $(\invertediota E_2)$ are too weak relative to \rf{C}. In Dummett's terminology the rules are \emph{not stable}. We'd expect further elimination rules to balance \rf{C} and permit the derivation of new consequences that do have something to do with the meaning of $\invertediota$ rather than just positive free logic without it. There are, of course, new consequences to do with identity. In particular, using \rf{C} as the major premise of $(=E)$, formulas of the form $A_{\invertediota x(x=t)}^y$ can be simplified to $A_t^y$. A number of philosophical questions arise. Is this misbalance in the rules problematic for an account in which the meaning of $\invertediota$ is supposed to be defined by the rules of inference governing it? If it is, what would a suitable additional elimination rule be? Or rather, as from a proof-theoretic perspective rules are preferable to axioms, what would additional introduction and elimination rules for $\invertediota$ be that permit us to prove \rf{C}? Is there a misbalance between the rules for $\invertediota$ and $=$? If so, how to balance it out? I shall leave these questions as suggestions for further work. 

In the light of the previous restrictions on $\invertediota$ rules, it is worth observing that if we add the rule $(\exists i)$ of section \ref{SecFDA}, \rf{C} can be restricted to atomic $t$. Choosing a fresh parameter $a$, by $(=E)$ $a=\invertediota xA, \invertediota x(x=a)=a\vdash \invertediota x(x=\invertediota xA)=\invertediota xA$, hence by $(\exists i)$ and the fact that \rf{C} follows from no premises, $\vdash \invertediota x(x=\invertediota xA)=\invertediota xA$. It can then also be established that \rf{C} need never be the major premise of $(=E)$. To illustrate with one case, instead of using $\invertediota x(x=t)=t$, take a fresh parameter $a$ to derive $a=t, \invertediota x(x=t)=t\vdash \invertediota x(x=t)=a$, then by two applications of $(=E)$ from $A_{\invertediota x(x=t)}^y$ derive $A_t^y$, discharge $a=t$ by $(\exists i)$. But \rf{C} has been used as a minor premise directly above a major premise of $(=E)$, so this looks to me like cheating and I am not sure about the significance of the result.

\subsection{FD2}
Recall \rf{FD2}: $\invertediota xA=t\leftrightarrow \forall x (x=t\leftrightarrow \forall y(A_y^x\leftrightarrow y=x))$. One half of \rf{FD2} follows from \rf{LA}: 

\begin{lemma}\label{halfofLA2}
$\rf{LA} \vdash \invertediota xA=t\rightarrow \forall x (x=t\leftrightarrow \forall y(A_y^x\leftrightarrow y=x))$
\end{lemma}

\begin{proof}
Where $a$ is a fresh parameter, positive free logic gives \rf{LA}, $\invertediota xA=a, \exists !a\vdash \forall y(A_y^x\leftrightarrow y=a))$ and $\invertediota xA=t, a=t\vdash \invertediota xA=a$, so \rf{LA}, $\invertediota xA=t, a=t, \exists !a\vdash \forall y(A_y^x\leftrightarrow y=a))$. Conversely, \rf{LA}, $\exists !a, \forall y(A_y^x\leftrightarrow y=a))\vdash \invertediota xA=a$, and also $\invertediota xA=a, \invertediota xA=t\vdash a=t$, hence \rf{LA}, $\invertediota xA=t, \exists !a, \forall y(A_y^x\leftrightarrow y=a))\vdash a=t$. The result follows by $(\leftrightarrow I)$ and $(\forall I)$.
\end{proof}

\noindent This means that formalising $\mathbf{FD2}$ in natural deduction requires no further elimination rules besides $(\invertediota E_1)$ and $(\invertediota E_2)$. It only requires either further introduction rules or a strengthening of $(\invertediota I)$ that permits the deduction of its other half: 

\lbp{FD2rl}{$FD2^{rl}$}{$\forall x (x=t\leftrightarrow \forall y(A_y^x\leftrightarrow y=x))\rightarrow \invertediota xA=t$}

\noindent Transposing this half of \rf{FD2} into a rule of natural deduction produces the following monstrosity:\footnote{A referee observes that the axiom, too, is a monstrosity.} 

{\small
\begin{prooftree}
\def\defaultHypSeparation{\hskip .1in}
\AxiomC{$[a=t]^i \ [b=a]^j \ [\exists !a]^k \ [\exists !b]^l$}
\noLine
\UnaryInfC{$\Sigma$}
\noLine
\UnaryInfC{$A_b^x$}
\AxiomC{$[a=t]^i \ [A_b^x]^m \ [\exists !a]^k \ [\exists !b]^l$}
\noLine
\UnaryInfC{$\Pi$}
\noLine
\UnaryInfC{$b=a$}
\AxiomC{$[\forall y(A_y^x\leftrightarrow y=x))]^n\ [\exists !a]^k$}
\noLine
\UnaryInfC{$\Xi$}
\noLine
\UnaryInfC{$a=t$}
\RightLabel{$_{i, j, k, l, m, n}$}
\TrinaryInfC{$\invertediota xA=t$}
\end{prooftree}
}

\noindent Although it is difficult to spell out the notion of harmony precisely, it is an educated guess that most philosophers working in the field will agree that this rule is not in harmony with $(\invertediota E_1)$ and $(\invertediota E_2)$. The problem is the nested biconditional of \rf{FD2rl}, which is responsible for the discharge of $\forall y(A_y^x\leftrightarrow y=a)$ above the deduction of $a=t$. I see no way of avoiding this. This is unfortunate for those who were hoping to develop a proof-theoretic semantics for $\mathbf{FD2}$, fortunate for those wish to exclude it. And indeed: the motivation for the chosen object theory is decidedly model-theoretic. That it may have to be set aside from the proof-theoretic perspective may thus not come as a surprise.

\section{Alternative Theories of Definite Descriptions}\label{alttheoryDD}
\subsection{FDA}\label{SecFDA}
With $(\invertediota I')$ instead of $(\invertediota I)$, the restriction to a single occurrence of an $\invertediota$ term can no longer be established as in Lemma \ref{iotarestricted}, as $(=I^{nG})$ requires the existence premise that is lacking. This can, however, be achieved by an alternative rule equivalent to one given by Garson \citeyearpar[259f]{garsonmodallogic}: 

\begin{prooftree}
\AxiomC{$[a=t]^i$}
\noLine
\UnaryInfC{$\Pi$}
\noLine
\UnaryInfC{$C$}
\RightLabel{$_i$}
\LeftLabel{$(\exists i) \ $}
\UnaryInfC{$C$}
\end{prooftree}

\noindent where $a$ is neither free in $t$, nor in $C$, nor in any undischarged assumptions it depends on except those of the assumption class of $a=t$.\bigskip 
 
\noindent Garson requires this rule to prove the completeness of some systems of quantified modal logic. It is not required for the completeness of $\mathbf{CPF}$ relative to its standard semantics. The following is the case: 

\begin{theorem}\label{existsi}
$(\exists i)$ is (a) admissible but (b) not derivable in $\mathbf{CPF}^{\invertediota}$. 
\end{theorem}

\begin{proof}
(a) That $(\exists i)$ is admissible is easily seen: deleting any application of $(\exists i)$ and replacing $a$ by $t$ results in a correct deduction. (b) Proof sketch for those familiar with the dual domain semantics for $\mathbf{CPF}^{\invertediota}$. Consider a semantics that is like the standard one (outer domain $\mathfrak{D}$ and inner domain $\mathfrak{E}\subseteq\mathfrak{D}$ over which $\exists !$ is interpreted), but with an additional element $\ast$ not in the outer domain. The interpretation of constants and $\invertediota$ terms stays untouched, i.e. they are interpreted in the outer domain, but variables ranger over $\mathfrak{D}\cup\{\ast\}$. Then all rules of $\mathbf{CPF}^{\invertediota}$ preserve validity. But $a=t$ is false, if $v(a)=\ast$, for any closed term $t$, and so $(\exists i)$ does not preserve validity.\footnote{This proof resulted from a discussion with Daniel Skurt.} 
\end{proof}

\noindent \emph{A fortiori} $(\exists i)$ is admissible but not derivable in  $\mathbf{IPF}^{\invertediota}$.

$(\exists i)$ permits the restriction of $(\exists I)$ and $(\forall E)$ to atomic instantiating terms: 

\begin{lemma}\label{altrestrictionquantifiers}
Given $(\exists i)$, $(\exists I)$ and $(\forall E)$ may be restricted to atomic terms. 
\end{lemma} 

\begin{proof}
Replace an application of $(\exists I)$ in which $t$ is complex by the following: 

\begin{prooftree}
\AxiomC{$[a=t]^i$}
\AxiomC{$A_t^x$}
\LeftLabel{$_{(=E)}$}
\BinaryInfC{$A_a^x$}
\AxiomC{$[a=t]^i$}
\AxiomC{$\exists ! t$}
\LeftLabel{$_{(=E)}$}
\BinaryInfC{$\exists ! a$}
\LeftLabel{$_{(\exists I)}$}
\BinaryInfC{$\exists xA$}
\LeftLabel{$_{(\exists i)}$}
\RightLabel{$_i$}
\UnaryInfC{$\exists xA$}
\end{prooftree}

\noindent Similarly for $(\forall E)$. 
\end{proof}

\noindent The restrictions of the $\invertediota$ rules can also be established: 

\begin{lemma}\label{altiotarestricted}
Given $(=I)$, $(=E)$ and $(\exists i)$, $(\invertediota I')$, $(\invertediota E_1)$ and $(\invertediota E_2)$ can be restricted to the single occurrence of an $\invertediota$ term where it is displayed in the rules. 
\end{lemma} 

\begin{proof}
Modify the deductions that prove Lemma \ref{iotarestricted}: change applications of $(=I^{nG})$ to $(\exists i)$ by deleting its left premise and anything above it. 
\end{proof}

\noindent Let $\mathbf{IPF}^{\invertediota A}$ be the system that results from $\mathbf{IPF}{'}^{\invertediota}$ by replacing $(\invertediota I)$ by $(\invertediota I')$ and $(=I^{nG})$ by $(\exists i)$. Similarly for $\mathbf{CPF}^{\invertediota A}$. Permutative reduction procedures work for $(\exists i)$ as for $(=I^{nG})$, and the reduction procedures for the rules for $\invertediota$ are almost the same: observe that in the latter the leftmost premise of $(\invertediota I)$ is idle, and so that it is missing altogether in $(\invertediota I')$ makes no essential difference.  

\begin{theorem}
Deductions in $\mathbf{IPF}^{\invertediota A}$ and $\mathbf{CPF}^{\invertediota A}$ can be brought into normal form. 
\end{theorem}

\begin{proof}
\emph{Mutatis mutandis} the proofs proceed as for $\mathbf{IPF}{'}^{\invertediota}$ and $\mathbf{CPF}{'}^{\invertediota}$.
\end{proof}

\noindent $(\exists i)$ could have been used for the systems of the previous sections, too. It has the advantage over $(=I^{nG})$ that it cannot so easily be objected that it hides a maximal formula. However, as it is admissible but not derivable, while $(=I^{nG})$ is derivable, it seemed preferable to use the latter in the systems with $(\invertediota I)$, so that the formalised theory of definite descriptions is Lambert's, not a variation thereof. It is also not clear whether $(\exists i)$ really is preferable over $(=I^{nG})$ from the philosophical perspective. Its intuitive content is that for any object picked out by a definite description, I can introduce an \emph{ad hoc} proper name and reason with that instead of the description. In Kripke's words, I can baptise the object picked out by the description. Arguably I can only do that with things that exist: what doesn't exist cannot be baptised. So for philosophical reasons there remains a hidden existence assumption in $(\exists i)$ that is made explicit in $(=I^{nG})$. An examination of these matters must await another occasion.

\subsection{Lambert's Serendipitous Theory}
At one point Lambert writes that the following axiomatises $\mathbf{FD2}$ \citep[32]{lambertEThe}:

\lbp{FDS}{$FDS$}{$\invertediota xA=t\leftrightarrow \forall x (x=t\leftrightarrow (A\land \forall y(A_y^x\rightarrow y=t)))$}

\noindent Note the subtle difference from $\invertediota xA=t\leftrightarrow \forall x (x=t\leftrightarrow (A\land \forall y(A_y^x\rightarrow y=x)))$,\footnote{This is a slight reformulation of the axiom used by Lambert and van Fraassen \citeyearpar[237]{fraassenlambert} to show the equivalence between what is here called $\mathbf{FD2}$.} the latter being equivalent to \rf{FD2}. This is a typo. \rf{FD2} and \rf{FDS} are not equivalent. \rf{FDS}, however, does not face the difficulties of casting it into rules of natural deduction that \rf{FD2} faced. For the purposes of proof-theoretic semantics, Lambert's typo proves to be serendipitous. 

Instead of adding \rf{FDS} to $\mathbf{CPF}$, I'll consider a slightly different axiomatisation of the resulting theory. 

One half of \rf{FDS} follows from the left to right half of \rf{LA}: 

\begin{lemma}\label{halfofFDS}
$\rf{LA} \vdash \invertediota xA=t\rightarrow \forall x (x=t\leftrightarrow (A\land \forall y(A_y^x\rightarrow y=t)))$
\end{lemma}

\begin{proof}
By positive free logic \rf{LA}, $\invertediota xA=t, \exists !a, \vdash \forall y(A_y^x\leftrightarrow y=a)$, and so \rf{LA}, $\invertediota xA=t, \exists !a\vdash A_a^x\land \forall y(A_y^x\rightarrow y=a)$. By $(=E)$, \rf{LA}, $\invertediota xA=t, \exists !a, a=t\vdash A_a^x\land \forall y(A_y^x\rightarrow y=t)$. Conversely, $\exists !a, A_a^x\land \forall y(A_y^x\rightarrow y=t)\vdash a=t$, so by $(\leftrightarrow I)$, \rf{LA}, $\invertediota xA=t, \exists !a\vdash a=t\leftrightarrow (A_a^x\land \forall y(A_y^x\rightarrow y=t))$ and the result follows by $(\forall I)$ and $(\rightarrow I)$. 
\end{proof}

\noindent As $\exists !a, A_a^x\land \forall y(A_y^x\rightarrow y=t)\vdash a=t$ just by positive free logic, it is possible to weaken the antecedent of the right to left half of \rf{FDS}. The following suffices: 

\lbp{FDS'}{$FDS'$}{$\forall x (x=t\rightarrow (A\land \forall y(A_y^x\rightarrow y=t)))\rightarrow\invertediota xA=t$}

\begin{lemma}
\rf{FDS'} implies the right to left half of \rf{LA}.
\end{lemma}

\begin{proof}
$\forall y(A_z^x\leftrightarrow y=a), \exists !a\vdash A_a^x\land \forall y(A_z^x\rightarrow y=a)$. Apply $(=E)$ to derive $\forall y(A_z^x\leftrightarrow y=a), \exists !a, b=a\vdash A_b^x\land \forall y(A_z^x\rightarrow y=a)$, hence by $(\rightarrow I)$ and $(\forall I)$ $\forall y(A_z^x\leftrightarrow y=a), \exists!a\vdash \forall x(x=a\rightarrow (A\land \forall y(A_z^x\rightarrow y=a))$. Apply \rf{FDS'} to derive \rf{FDS'}, $\forall y(A_z^x\leftrightarrow y=a), \exists !a\vdash \invertediota xA=a$. The result follows by $(\rightarrow I)$ and $(\forall I)$. 
\end{proof}

\noindent Thus let $\mathbf{FDS}$ be the system that results from adding \rf{LA2} and \rf{FDS'} to $\mathbf{CPF}$. 

Next we'll show that $\mathbf{FDS}$ is weaker than $\mathbf{FD2}$. The right hand side of \rf{FD2} implies the right hand side of \rf{FDS}: 

\begin{lemma}\label{rhsFD2FDS}
$\forall x(x=t\leftrightarrow \forall y(A_y^x\leftrightarrow y=x))\vdash \forall x(x=t\leftrightarrow (A\land \forall y(A_y^x\rightarrow y=t)))$
\end{lemma}

\begin{proof}
$\forall x(x=t\leftrightarrow \forall y(A_y^x\leftrightarrow y=x)), \exists !a\vdash a=t\leftrightarrow \forall y(A_y^x\leftrightarrow y=a)$. Furthermore $\forall y(A_y^x\leftrightarrow y=a), \exists !a\vdash A_a^x$ and $\forall y(A_y^x\leftrightarrow y=a), a=t\vdash \forall y(A_y^x\rightarrow y=t)$, hence (1) $\forall x(x=t\leftrightarrow \forall y(A_y^x\leftrightarrow y=x)), \exists !a, a=t\vdash A_a^x\land \forall y(A_y^x\rightarrow y=t)$. Also (2) $\exists !a, A_a^x\land \forall y(A_y^x\rightarrow y=t)\vdash a=t$. Hence by $(\leftrightarrow I)$ from (1) and (2) $\forall x(x=t\leftrightarrow \forall y(A_y^x\leftrightarrow y=x)), \exists !a\vdash a=t\leftrightarrow (A_a^x\land \forall y(A_y^x\rightarrow y=t)$, and so by $(\forall I)$, $\forall x(x=t\leftrightarrow \forall y(A_y^x\leftrightarrow y=x))\vdash \forall x(x=t\leftrightarrow (A\land \forall y(A_y^x\rightarrow y=t))$.
\end{proof}

\begin{lemma}\label{FD2toFDS}
$\rf{FD2}\vdash \rf{FDS}$
\end{lemma}

\begin{proof}
From lemmas \ref{rhsFD2FDS} and \ref{halfofFDS}. 
\end{proof}

\noindent The two systems are close. The right hand side of \rf{FDS} implies the left to right direction of the right hand side of \rf{FD2}: 

\begin{lemma}
$\forall x(x=t\leftrightarrow (A\land\forall y(A_y^x\rightarrow y=t)))\vdash \forall x(x=t\rightarrow \forall y(A_y^x\leftrightarrow y=x))$
\end{lemma}

\begin{proof}
$\forall x(x=t\leftrightarrow (A\land\forall y(A_y^x\rightarrow y=t))), \exists !a\vdash a=t\leftrightarrow (A_a^x\land\forall y(A_y^x\rightarrow y=t)$, so $\forall x(x=t\leftrightarrow (A\land\forall y(A_y^x\rightarrow y=t)), \exists !a, a=t\vdash A_a^x\land\forall y(A_y^x\rightarrow y=t)$. Use $(=E)$ to derive $\forall x(x=t\leftrightarrow (A\land\forall y(A_y^x\rightarrow y=t), \exists !a, a=t\vdash A_a^x\land\forall y(A_y^x\rightarrow y=a)$. As $A_a^x\land\forall y(A_y^x\rightarrow y=a)\vdash \forall y(A_y^x\leftrightarrow y=x))$, the result follows. 
\end{proof}

\noindent  But they are not the same. The right hand side of \rf{FDS} does not imply the right to left direction of the right hand side of \rf{FD2}: 

\begin{lemma}\label{FDSnottoFD2}
$\forall x(x=t\leftrightarrow (A\land\forall y(A_y^x\rightarrow y=t)))\not\vdash \forall x(\forall y(A_y^x\leftrightarrow y=x)\rightarrow x=t)$
\end{lemma}

\begin{proof}
Sketch for those familiar with the dual domain semantics for $\mathbf{CPF}$. Suppose $t$ refers to an object in the outer domain, and there is exactly one $A$ in the inner domain. Then $\forall x(\forall y(A_y^x\leftrightarrow y=x)\rightarrow x=t)$ is false. $\forall y(A_y^x\rightarrow y=t)$ is also false. Whenever $x$ is assigned an object in the inner domain, $x=t$ is not satisfied. Hence $x=t\leftrightarrow (A\land\forall y(A_y^x\rightarrow y=t))$ is satisfied, whenever $x$ is assigned an object in the inner domain, and hence $\forall x(x=t\leftrightarrow (A\land\forall y(A_y^x\rightarrow y=t)))$ is true. 
\end{proof}

\begin{theorem}\label{FDSsubFD2}
If $\Gamma\vdash_\mathbf{FDS} A$, then $\Gamma\vdash_\mathbf{FD2} A$, but not conversely.
\end{theorem}

\begin{proof}
From lemmas \ref{FD2toFDS} and \ref{FDSnottoFD2}. 
\end{proof}

\noindent \rf{FDS'} is easily cast into rule form:  

\begin{prooftree} 
\AxiomC{$[a=t]^i \ [\exists ! a]^k$}
\noLine
\UnaryInfC{$\Xi$}
\noLine
\UnaryInfC{$F_a^x$}
\AxiomC{$[a=t]^j \ [\exists !a]^k \ [F_b^x]^l \ [\exists !b]^m$}
\noLine
\UnaryInfC{$\Pi$}
\noLine
\UnaryInfC{$b=t$}
\LeftLabel{$(\invertediota I^S)$ \ }
\RightLabel{$_{i, j, k, l, m}$}
\BinaryInfC{$\invertediota xF=t$}
\end{prooftree}

\noindent where $a$ and $b$ are not free in $F$ nor in $t$ and does not occur in any undischarged assumptions in $\Xi$ and $\Pi$ except those in the assumption classes displayed.

\begin{lemma}\label{FDSderivable}
(a) \rf{FDS'} and $(\invertediota I^S)$ are interderivable. (b) The right to left half of \rf{LA} is derivable from $(\invertediota I^S)$, the other half from $(\invertediota E_1)$ and $(\invertediota E_2)$. 
\end{lemma}

\begin{proof}
Exercise. 
\end{proof}

\noindent As for $\mathbf{CPF}^{\invertediota A}$, we need Garson's $(\exists i)$ to prove normalisation. Theorem \ref{existsi} goes through as before.\footnote{In the dual domain semantics for $\mathbf{FD2}$, the outer domain is a singleton: all non-referring terms refer to the same object, i.e. $\neg\exists !t\land\neg\exists !s\rightarrow s=t$ is provable from \rf{FD2}.} So does Lemma \ref{altrestrictionquantifiers}. Concerning Lemma \ref{altiotarestricted}, it suffices to consider $(\invertediota I^S)$, the elimination rules staying put. If $t$ is a complex term, replace it by a fresh parameter $c$, deduce $\invertediota xF=t$ from $\invertediota xF=c$ and $c=t$, then apply $(\exists i)$ to discharge the latter. 

We therefore get systems with the same neat properties as $\mathbf{IPF}^{\invertediota A}$ and $\mathbf{CPF}^{\invertediota A}$ if we replace $(\invertediota I')$ by $(\invertediota I^S)$, and, as before, restrict occurrences of $\invertediota$ terms in $\invertediota$ rules to where they are displayed in the rules. Call the resulting systems $\mathbf{IPF}^{\invertediota S}$ and $\mathbf{CPF}^{\invertediota S}$. 

\begin{theorem}
$\Gamma\vdash_{\mathbf{CPF}^{\invertediota S}} A$ iff $\Gamma\vdash_\mathbf{FDS} A$.
\end{theorem} 

\begin{proof}
From lemma \ref{FDSderivable}. 
\end{proof}

\begin{theorem}
(a) If $\Gamma\vdash_{\mathbf{CPF}^{\invertediota A}} A$, then $\Gamma\vdash_{\mathbf{CPF}^{\invertediota 'S}} A$. (b)  If $\Gamma\vdash_{\mathbf{IPF}^{\invertediota A}} A$, then $\Gamma\vdash_{\mathbf{IPF}^{\invertediota 'S}} A$. 
\end{theorem}

\begin{proof}
$(\invertediota I')$ is the special case of $\invertediota I^S)$ where $a$ and $b$ are the same parameter and $a=t$ above the right premise discharged vacuously. Replacing $(\invertediota I')$ by $(\invertediota I^S)$ therefore results in a stronger system, as $(\invertediota I^S)$ permits the discharge of more assumptions, and so consequences are derivable with it that are not derivable without it. 
\end{proof}

\noindent Now for what we are for, the normalisation theorem. 

\begin{theorem}
Deductions in $\mathbf{IPF}^{\invertediota S}$ and $\mathbf{CPF}^{\invertediota S}$ can be brought into normal form. 
\end{theorem}

\begin{proof}
The new rule being an introduction, it suffices to consider maximal formulas that are concluded by $(\invertediota I^S)$ and eliminated by $(\invertediota E_1)$ and $(\invertediota E_2)$. The former case is as for $\mathbf{CPF}^{\invertediota A}$. The latter case goes as follows: 

\begin{center}
{\scriptsize{
\AxiomC{$[a=t]^i \ [\exists ! a]^k$}
\noLine
\UnaryInfC{$\Xi$}
\noLine
\UnaryInfC{$F_t^x$}
\AxiomC{$[a=t]^j \ [\exists !a]^k \ [F_b^x]^l \ [\exists !b]^m$}
\noLine
\UnaryInfC{$\Pi$}
\noLine
\UnaryInfC{$b=t$}
\LeftLabel{$_{(\invertediota I^S)}$}
\RightLabel{$_{i, j, k, l, m}$}
\BinaryInfC{$\invertediota xF=t$}
\AxiomC{$\Sigma_1$}
\noLine
\UnaryInfC{$F_u^x$}
\AxiomC{$\Sigma_2$}
\noLine
\UnaryInfC{$\exists !t$}
\AxiomC{$\Sigma_3$}
\noLine
\UnaryInfC{$\exists !u$}
\LeftLabel{$_{(\invertediota E_2)}$}
\QuaternaryInfC{$u=t$}
\DisplayProof\quad $\leadsto$ 
\AxiomC{$\overline{ \ t=t \ } \ 
\mathbin{\stackon[6pt]{[\exists !t]}{\Sigma_2}} \ 
\mathbin{\stackon[6pt]{[F_u^x]}{\Sigma_1}} \ 
\mathbin{\stackon[6pt]{[\exists !u]}{\Sigma_3}}$}
\noLine
\UnaryInfC{$\Pi{_t^a}{_u^b}$}
\noLine
\UnaryInfC{$u=t$}
\DisplayProof
}}
\end{center}

\noindent The proof now proceeds \emph{mutatis mutandis} as in the previous cases. 
\end{proof}

\noindent $\mathbf{IPF}^{\invertediota S}$ and $\mathbf{CPF}^{\invertediota S}$ raise questions concerning harmony. They have the same elimination rules for $\invertediota$ as do $\mathbf{IPF}{'}^{\invertediota}$ and $\mathbf{CPF}{'}^{\invertediota}$, but a different introduction rule. There is a thought in proof-theoretic semantics that one and the same collection of elimination rules should not be harmonious with different, non-equivalent introduction rules (and analogously for introduction rules). 

The reason the normalisation procedure works for $\mathbf{IPF}^{\invertediota S}$ and $\mathbf{CPF}^{\invertediota S}$ is that $\vdash t=t$. Even if $t=t$ required $\exists !t$, as in negative free logic, the procedure would work, as the latter is given as a premise. Once more, questions concerning the relation between $\invertediota$ and $=$ arise, and the dependence of the meaning of $\invertediota$ on that of $=$. However, it strikes me as if an appeal to rules for $=$ in the normalisation procedure can hardly be objectionable, as the rules for $\invertediota$ appeal to $=$. That the meaning of the former can only be defined in terms of the latter supplies no valid objection to defining its meaning in terms of rules of inference.

\subsection{Hintikka's System}
Hintikka and Lambert proposed similar theories of definite descriptions around the same time. To honour them both with an axiom, I shall name the following \emph{Hintikka's Axiom}:\footnote{This slightly, but crucially, different from the axiom Hintikka proposed in \citep{hintikkatowardsdefdesc}. Lambert criticised that proposal on the grounds that it leads to contradiction \citep{lambertnotesEIII}. The axiom Hintikka appealed to there is in our notation: $(\ast )$ \ $\invertediota xA=t\leftrightarrow (A_t^x\land \forall y(A_y^x\rightarrow y=t))$. In free logic, $A_t^x\land \forall y(A_y^x\rightarrow y=t)$ is not equivalent to $\forall y(A_y^x\leftrightarrow y=t)$, so \rf{HA} is safe. Lambert's proof that $(\ast)$ leads to contradiction appeals to the law of identity, $(=I)$. Hintikka responded by rejecting $t=t$ for definite descriptions on the grounds that if there is more than one $F$, different occurrences of $\invertediota xF$ may well refer to different objects \citep{hintikkaDDandidentity}. Lambert in turn responded to Hintikka with the proposal to formalise a theory of definite descriptions with the single axiom \rf{FD2} \citep{lambertdefdescrandid}, which he had already proposed in \citep[87]{lambertnotesEIV}.}

\lbp{HA}{$HA$}{$\invertediota xA=t\leftrightarrow \forall y(A_y^x\leftrightarrow y=t)$}

\noindent As $t$ is any term, \rf{HA} requires fewer existence assumptions than \rf{LA} in its use in deductions, because $t$ need not exist, and so to formalise Hintikka's theory of definite descriptions in natural deduction, a correspondingly weakened elimination rule is required: 

\begin{prooftree}
\AxiomC{$\invertediota xF=t$}
\AxiomC{$F_u^x$}
\AxiomC{$\exists !u$}
\LeftLabel{$(\invertediota E_2') \ $}
\TrinaryInfC{$u=t$}
\end{prooftree}

\begin{lemma}
$(\invertediota I')$, $(\invertediota E_1)$, $(\invertediota E_2')$ and \rf{HA} are interderivable. 
\end{lemma}

\begin{proof}Exercise.\end{proof}

\noindent $t$ and $u$ can be restricted to atomic terms, for the same reasons as in the case of $\mathbf{IPF}^{\invertediota A}$ and $\mathbf{CPF}^{\invertediota A}$. Call the systems that result by replacing $(\invertediota E_2)$ by $(\invertediota E_2')$ in them $\mathbf{IPF}^{\invertediota H}$ and $\mathbf{CPF}^{\invertediota H}$. 

\begin{theorem}
Deductions in $\mathbf{IPF}^{\invertediota H}$ and $\mathbf{CPF}^{\invertediota H}$ can be brought into normal form. 
\end{theorem}

\begin{proof}
The reduction procedure for maximal formulas arising from $(\invertediota I')$ and $(\invertediota E_2')$ is evidently:\bigskip

\AxiomC{$[a=t]^i \ [\exists ! a]^k$}
\noLine
\UnaryInfC{$\Xi$}
\noLine
\UnaryInfC{$F_a^x$}
\AxiomC{$[F_a^x]^j \ [\exists !a]^k$}
\noLine
\UnaryInfC{$\Pi$}
\noLine
\UnaryInfC{$a=t$}
\LeftLabel{$_{(\invertediota I')}$}
\RightLabel{$_{i, j, k}$}
\BinaryInfC{$\invertediota xF=t$}
\AxiomC{$\Sigma_1$}
\noLine
\UnaryInfC{$F_u^x$}
\AxiomC{$\Sigma_2$}
\noLine
\UnaryInfC{$\exists !u$}
\LeftLabel{$_{(\invertediota E_2')}$}
\TrinaryInfC{$u=t$}
\DisplayProof\quad $\leadsto$ 
\AxiomC{$\mathbin{\stackon[6pt]{[F_u^x]}{\Sigma_1}} \ 
\mathbin{\stackon[6pt]{[\exists !u]}{\Sigma_2}}$}
\noLine
\UnaryInfC{$\Pi{_u^a}$}
\noLine
\UnaryInfC{$u=t$}
\DisplayProof

\bigskip

\noindent The proof now proceeds \emph{mutatis mutandis} as in the previous cases. 
\end{proof}

\noindent We are now faced with a rather interesting situation. Only two formulas are discharged  above the right premise of $(\invertediota I')$, and these correspond precisely to the two minor premises of $(\invertediota E_2')$. An assumption corresponding to the third premise of $(\invertediota E_2)$, however, could be discharged above the right premise of $(\invertediota I')$, as $\exists !t$ is derivable from $\exists !a$ and $a=t$. The theories of definite descriptions of $\mathbf{IPF}^{\invertediota A}$ and $\mathbf{CPF}^{\invertediota A}$ and of $\mathbf{IPF}^{\invertediota H}$ and $\mathbf{CPF}^{\invertediota H}$ are not the same. The instance of \rf{HA} for a term $t$ is not derivable from \rf{LA} unless $t$ exists, which is the additional assumption in $(\invertediota E_2)$. This points to a disharmony somewhere. The question is where. Which set of rules is harmonious, those of $\mathbf{IPF}^{\invertediota A}$ and $\mathbf{CPF}^{\invertediota A}$ or those of $\mathbf{IPF}^{\invertediota H}$ and $\mathbf{CPF}^{\invertediota H}$? 

One may be inclined to solve this conundrum by declaring that so much the worse for both: neither set of rules satisfies the criteria of proof-theoretic semantics. That option, however, isn't really on the table. There is the thought that introduction rules can be laid down as we please: they are self-justifying, as Dummett says \citep[251]{dummettLBM}. Harmony determines the elimination rules. Others prefer to see matters the other way round and treat elimination rules as self-justifying, and introduction rules determined by harmony from them. 

I shan't attempt to address this question here, and merely note it as a further philosophical issue that arises for a proof-theoretic semantics for the $\invertediota$ operator.

\subsection{Relations between the Systems}
It remains to say something about the relations between the systems. 

\begin{lemma}\label{FDAsubFDS}
(a) \rf{FDS} $\vdash$ \rf{LA1}. (b) \rf{FDS} $\vdash$ \rf{LA2}.
\end{lemma}

\begin{proof}
(a) Standard reasoning in $\mathbf{CPF}$ gives $\forall y(A\leftrightarrow y=t), \exists !a, a=t\vdash A_a^x$ and $\forall y(A\leftrightarrow y=t)\vdash \forall y(A\rightarrow y=t)$, hence $\forall y(A\leftrightarrow y=t), \exists !a, a=t\vdash A_a^x\land\forall y(A\rightarrow y=t)$. Furthermore, $A_a^x\land\forall y(A\rightarrow y=t), \exists !a\vdash a=t$, hence $\forall y(A\leftrightarrow y=t)\vdash \forall x(x=t\leftrightarrow (A\land\forall y(A\leftrightarrow y=t)))$. So \rf{FDS}, $\forall y(A\leftrightarrow y=t)\vdash\invertediota xA=t$, and the result follows. 

(b) \rf{FDS}, $\invertediota xA=a\vdash \forall x(x=a\leftrightarrow (A\land\forall y(A\leftrightarrow y=a)))$, hence \rf{FDS}, $\invertediota xA=a, \exists !a\vdash a=a\leftrightarrow (A_a^x\land\forall y(A\leftrightarrow y=a)))$, so \rf{FDS}, $\invertediota xA=a, \exists !a\vdash A_a^x\land\forall y(A\leftrightarrow y=a)$. The result follows by familiar moves in $\mathbf{CPF}$. 
\end{proof}

\noindent The following relationships of inclusion have been established: 

\begin{theorem}\label{relations}
(a) $\mathbf{CPF}{'}^{\invertediota}\subset\mathbf{CPF}{'}^{\invertediota 1}\subseteq\mathbf{CPF}^{\invertediota A}\subseteq\mathbf{CPF}^{\invertediota S}\subseteq\mathbf{FD2}$

(b) $\mathbf{CPF}{'}^{\invertediota}\subset\mathbf{CPF}{'}^{\invertediota 1}\subseteq\mathbf{CPF}^{\invertediota A}\subseteq\mathbf{CPF}^{\invertediota H}$
\end{theorem}

\begin{proof}
(a) $\mathbf{CPF}{'}^{\invertediota}\subset\mathbf{CPF}{'}^{\invertediota 1}$ is a result of Lambert's. $\mathbf{CPF}{'}^{\invertediota 1}\subseteq\mathbf{CPF}^{\invertediota A}$ because $\mathbf{CPF}^{\invertediota A}\vdash \invertediota x(x=t)=t$. $\mathbf{CPF}^{\invertediota A}\subseteq\mathbf{CPF}^{\invertediota S}$ is shown in Lemma \ref{FDAsubFDS}.

(b) Because $\mathbf{CPF}^{\invertediota H}\vdash \invertediota x(x=t)=t$. 
\end{proof}

\begin{corollary}
The analogous relationships hold for the intuitionist systems.
\end{corollary}

\begin{proof}
Because each intuitionist system is included in the corresponding classical one.
\end{proof}

\noindent It is reasonable to conjecture the following: 

\begin{conjecture}
(a) All inclusions of Theorem \ref{relations} are proper. 

(b) $\mathbf{CPF}^{\invertediota H}$ and $\mathbf{CPF}^{\invertediota S}$ are incomparable. 

(c) $\mathbf{CPF}^{\invertediota H}$ and $\mathbf{FD2}$ are incomparable.
\end{conjecture}

\noindent The reason why (a) is reasonable is that $(\invertediota I^S)$ of $\mathbf{CPF}^{\invertediota S}$ permits the deduction of $\invertediota xFx=t$ from fewer undischarged assumptions than $(\invertediota I')$ of $\mathbf{CPF}^{\invertediota H}$, but $(\invertediota E_2')$ of $\mathbf{CPF}^{\invertediota H}$ permits the deduction of more formulas $u=t$ from $\invertediota xFx=t$ than $(\invertediota E_2)$ of $(\invertediota I^S)$. (b) is reasonable for exactly parallel reasons. Further investigations into the relationships of the systems presented here must await another occasion, but the following is worth recoding: 

\begin{lemma}\label{FDSsubHA}
The rhs of \rf{HA} implies the rhs of \rf{FDS}.
\end{lemma}

\begin{proof}
Assume $\forall y(A_y^x\leftrightarrow y=t)$ and $\exists !a$. (1) Assume $a=t$. Then $A_a^x\leftrightarrow a=t$ by $(\forall E)$, hence $A_a^x$, and also $\forall y(A_y^x\rightarrow y=t)$. Hence $A_a^t\land\forall y(A_y^x\rightarrow y=t)$. (2) Assume $A_a^x\land\forall y(A_y^x\rightarrow y=t)$. Then $A_a^x\rightarrow a=t$ by  $(\land E)$ and $(\forall E)$, so $a=t$ by $(\rightarrow E)$. So by $(\leftrightarrow I)$ $a=t\leftrightarrow (A_a^x\land\forall y(A_y^x\rightarrow y=t))$. Hence $\forall x(x=t\leftrightarrow (A\land\forall y(A_y^x\rightarrow y=t)))$, discharing $\exists !a$. 
\end{proof}

\section{Concluding Philosophical Remarks}
Positive free logic is the preferred logic of many free logicians. It is decidedly unRussellian, while negative free logic remains near to the Russellian paradigm. The proof theory of positive free logic is simpler than that of negative free logic. It is in this sense closer to non-free classical and intuitionist logic in that it merely restricts the rules for the quantifiers. 

Proving a normalisation theorem for $\mathbf{IPF}$ and $\mathbf{CPF}$ was straightforward, but when it comes to definite descriptions, the systems had to be modified by adding the rule $(=I^{nG})$ and normalisation was proved for $\mathbf{IPF}{'}^{\invertediota}$ and $\mathbf{CPF}{'}^{\invertediota}$. This meant that $(\forall E)$ and $(\exists I)$ can be restricted to atomic instantiating terms and furthermore, the $\invertediota$ rules can be restricted to a single occurrence of an $\invertediota$ term where it is displayed in the rules, all other terms being atomic. This may be of philosophical relevance, if the rules governing $\invertediota$ are to define its meaning. 

Considerations from proof-theoretic semantics suggested a misbalance between the introduction and elimination rules for the $\invertediota$ operator that formalises Lambert's minimal theory of definite descriptions in positive free logic, i.e. $\mathbf{IPF}{'}^{\invertediota}$ and $\mathbf{CPF}{'}^{\invertediota}$. The problem is solved by dropping an existence premise from $(\invertediota I)$. 

The rule $(\invertediota I')$ is arguably in perfect harmony with $(\invertediota E_1)$ and $(\invertediota E_2)$. Normalisation theorems were proved for $\mathbf{IPF}^{\invertediota A}$ and $\mathbf{CPF}^{\invertediota A}$ in which $(=I^{nG})$ is replaced by $(\exists i)$, which incidentally drops an existence premise from the former, just as $(\invertediota I')$ drops an existence premise from $(\invertediota I)$. The result is a new theory of definite descriptions conforming to principles of proof-theoretic semantics. 

The matter is not so easily settled, however, and proof-theoretic considerations raise a number of questions. 

It was observed that normalisation theorems are provable for $\mathbf{IPF}{'}^{\invertediota 1}$ and $\mathbf{CPF}{'}^{\invertediota 1}$, arising from $\mathbf{IPF}{'}^{\invertediota}$ and $\mathbf{CPF}{'}^{\invertediota}$ by adding \rf{C} as an axiom. These systems, equivalent to Lambert's $\mathbf{FD1}$, therefore satisfy a necessary condition for proof-theoretic semantics. But the rules exhibit a misbalance, as \rf{C} affords grounds for deriving formulas with $\invertediota$ to which there are no corresponding elimination rules. The question arises whether alternative rules can be formulated that rectify this and rebalance the rules. \rf{C} is provable in $\mathbf{IPF}^{\invertediota A}$ and $\mathbf{CPF}^{\invertediota A}$, but these systems are stronger than $\mathbf{IPF}{'}^{\invertediota 1}$ and $\mathbf{CPF}{'}^{\invertediota 1}$. 

No proof-theoretically satisfactory rules have been found for Lambert's $\mathbf{FD2}$, which formalises a chosen object theory of improper definite descriptions and `non-denoting' names. It was suggested that this is fair enough, as the motivation for this theory is model-theoretic in spirit: there is, after all, a crucial claim concerning reference at its heart.  

Four further theories of definite descriptions were considered. $\mathbf{IPF}^{\invertediota S}$ and $\mathbf{CPF}^{\invertediota S}$ have as their source a serendipitous typo in a paper of Lambert's.  $\mathbf{IPF}^{\invertediota H}$ and $\mathbf{CPF}^{\invertediota H}$ have their origin in a charitable reading of an axiom of Hintikka's. 

In the latter, an existence assumption is dropped from $(\invertediota E_2)$ of $\mathbf{IPF}^{\invertediota A}$ and $\mathbf{CPF}^{\invertediota A}$. The puzzle now is that the collection of rules $(\invertediota I'), (\invertediota E_1), (\invertediota E_2)$ of $\mathbf{IPF}^{\invertediota A}$ and $\mathbf{CPF}^{\invertediota A}$ and the collection $(\invertediota I'), (\invertediota E_1), (\invertediota E_2')$ of $\mathbf{IPF}^{\invertediota H}$ and $\mathbf{CPF}^{\invertediota H}$ appear to be equally harmonious. No significant difference can be discerned, as the dropped existence assumption of $(\invertediota E_2')$ corresponds to a discharged hypothesis above the right hand premise of $(\invertediota I')$ that it is unnecessary to mention explicitly, because it is derivable from the other two assumptions discharged there. 

$(\invertediota I^S)$ of $\mathbf{IPF}^{\invertediota S}$ and $\mathbf{CPF}^{\invertediota S}$ permits a more complex pattern of discharge than $(\invertediota I')$. Yet this does not appear to upset harmony in any way. Adding a corresponding premises to $(\invertediota E_2)$ would amount to adding $t=t$, which is redundant. As before, omissions for redundancies can hardly upset harmony. 

These observations point, I think, to an issue that has been underinvestigated in proof-theoretic semantics: the discharge of assumptions. The question is how to accommodate different patterns of discharge, while premises and conclusions of rules stay the same, within the notion of harmony. As far as I can see there is no literature at all on this issue, although it is of some importance, seeing that classical logic requires rules that discharge assumptions with negation or implication as main operator.  

It is my contention that these considerations open up a fruitful field of enquiry for the proof-theoretic semantics for $\invertediota$ and term-forming operators, and, indeed, proof-theoretic semantics in general. 

\bigskip

\noindent\textbf{Acknowledgements.} As usual when writing on this topic, I must thank Andrzej Indrzejczak for his insights and comments. 

\bigskip

\noindent\textbf{Funding.} The research in this paper was funded by the European Union (ERC, ExtenDD, project number: 101054714). Views and opinions expressed are however those of the author(s) only and do not necessarily reflect those of the European Union or the European Research Council. Neither the European Union nor the granting authority can be held responsible for them.

\bigskip

\setlength{\bibsep}{0pt}
\bibliographystyle{chicago}
\bibliography{KurbisNormalPositiveFreeLogic}

\end{document}